\documentclass[10pt]{article}
\usepackage[T1]{fontenc}
\usepackage[utf8]{inputenc}
\usepackage{lmodern}
\usepackage[a4paper,margin=1in]{geometry}
\usepackage{amsmath,amssymb,amsthm,mathtools,mathrsfs}
\usepackage{microtype}
\usepackage[skip=0.45\baselineskip plus 2pt,indent=0pt]{parskip}
\usepackage{enumitem}
\usepackage{tikz}
\usetikzlibrary{arrows.meta,calc,positioning,decorations.pathreplacing}
\usepackage[colorlinks=true,linkcolor=blue,citecolor=blue,urlcolor=blue]{hyperref}
\usepackage{aliascnt}
\usepackage[nameinlink,noabbrev,capitalise]{cleveref}
\allowdisplaybreaks[2]
\newtheorem{theorem}{Theorem}[section]
\newaliascnt{proposition}{theorem}
\newtheorem{proposition}[proposition]{Proposition}
\aliascntresetthe{proposition}
\newaliascnt{lemma}{theorem}
\newtheorem{lemma}[lemma]{Lemma}
\aliascntresetthe{lemma}
\newaliascnt{fact}{theorem}
\newtheorem{fact}[fact]{Fact}
\aliascntresetthe{fact}
\crefname{fact}{Fact}{Facts}
\Crefname{fact}{Fact}{Facts}
\newaliascnt{corollary}{theorem}

\aliascntresetthe{corollary}
\newaliascnt{claim}{theorem}
\newtheorem{claim}[claim]{Claim}
\aliascntresetthe{claim}
\crefname{claim}{Claim}{Claims}
\Crefname{claim}{Claim}{Claims}
\newaliascnt{definition}{theorem}
\newtheorem{definition}[definition]{Definition}
\aliascntresetthe{definition}
\newaliascnt{remark}{theorem}

\aliascntresetthe{remark}
\numberwithin{equation}{section}
\newcommand{\N}{\mathbb N}
\newcommand{\Q}{\mathbb Q}
\newcommand{\E}{\mathbb E}
\newcommand{\Prb}{\mathbb P}
\newcommand{\one}{\mathbf 1}
\newcommand{\bx}{\boldsymbol x}
\newcommand{\by}{\boldsymbol y}
\newcommand{\bz}{\boldsymbol z}
\newcommand{\bw}{\boldsymbol w}
\newcommand{\Fu}{\mathcal F}
\newcommand{\Gu}{\mathcal G}
\newcommand{\Hu}{\mathcal H}
\newcommand{\Du}{\mathcal D}
\newcommand{\Lu}{\mathcal L}
\newcommand{\Pu}[1]{\Pi_{\mathrm u,\infty}^{(#1)}}
\newcommand{\Pufin}[1]{\Pi_{\mathrm u,\mathrm{fin}}^{(#1)}}
\newcommand{\Pord}[1]{\Pi_{#1,\infty}}
\newcommand{\rev}{\operatorname{rev}}
\newcommand{\Col}{\operatorname{Col}}
\newcommand{\dd}{\,\mathrm d}
\newcommand{\eps}{\varepsilon}
\newcommand{\ind}[1]{\mathbf1_{\{#1\}}}
\setlist[enumerate]{label=\textup{(\roman*)},itemsep=2pt,topsep=4pt,leftmargin=2em}
\hypersetup{pdftitle={On the structure of uniform Turan densities},pdfsubject={Finite palette attainment and intervals of uniform hypergraph Turan densities},pdfauthor={Heng Li and Xizhi Liu}}
\title{\Large\bfseries On the structure of uniform Tur\'an densities}
\author{%
Heng~Li%
\and
Xizhi~Liu%
}
\date{\today}
\begin{document}
\maketitle
\begin{abstract}
Motivated by parallel developments in ordinary and $\ell$-degree Tur\'an problems, we study the set of $(r-2)$-uniform Tur\'an densities of possibly infinite families of $r$-uniform hypergraphs. We prove that, for all sufficiently large $r$, these densities exhibit a phase transition at $4r^{-r}$, from a countable set of algebraic values below this threshold to the full interval $[4r^{-r},1]$. For every $r\ge3$, the densities below the threshold are exactly the finite-palette Lagrangians in that range, and each is realized by a finite forbidden family. Adapting a method of Pikhurko, we show that, for every $r\ge3$, the set of uniform densities omits continuum many of its limit points below the threshold and is therefore not closed.
\end{abstract}

\section{Introduction}\label{sec:intro}
For an integer $r\ge3$, an \emph{$r$-uniform hypergraph}, or \emph{$r$-graph}, $H$ consists of a nonempty finite vertex set $V(H)$ and a collection of $r$-element subsets of $V(H)$, called edges. We identify a hypergraph with its edge set, and hence write $|H|$ for the number of edges of $H$. For a finite set $V$, let $\binom Vr$ denote the family of its $r$-element subsets. If $G$ is an $(r-2)$-graph on $V$, let $\mathcal K_r(G)\coloneqq\{e\in\binom Vr\colon\binom e{r-2}\subseteq G\}$. Given $d\in[0,1]$ and $\eta>0$, an $r$-graph $H$ on $n$ vertices is \emph{$(d,\eta,r-2)$-dense} if
\[
|H\cap\mathcal K_r(G)|\ge d|\mathcal K_r(G)|-\eta n^r
\]
for every $(r-2)$-graph $G$ on $V(H)$. For a possibly infinite family $\Fu$ of $r$-graphs, its \emph{$(r-2)$-uniform Tur\'an density} $\pi_{r-2}(\Fu)$ is the supremum of the values $d$ such that, for every $\eta>0$ and $n_0\in\N$, there exists an $\Fu$-free $(d,\eta,r-2)$-dense $r$-graph with at least $n_0$ vertices. Each forbidden hypergraph is assumed to have an edge; the empty forbidden family is allowed. We write
\[
\begin{aligned}
\Pu r&\coloneqq\{\pi_{r-2}(\Fu)\colon\Fu\text{ is a possibly infinite family of }r\text{-graphs}\},\\
\text{and}\qquad\Pufin r&\coloneqq\{\pi_{r-2}(\Fu)\colon\Fu\text{ is a finite family of }r\text{-graphs}\}.
\end{aligned}
\]

The study of uniform Tur\'an problems for $3$-graphs was initiated by Erd\H{o}s and S\'os~\cite{erdossos}; see Reiher~\cite{reiher-survey} for a survey of uniform density notions and the associated extremal problems. Glebov, Kr\'al', and Volec~\cite{gkv} proved that the $3$-graph on 4 vertices with 3 edges has uniform density $1/4$, and Reiher, R\"odl, and Schacht~\cite{rrs-quarter} gave a proof using hypergraph regularity. Buci\'c, Cooper, Kr\'al', Mohr, and Munh\'a Correia~\cite{bckmc-cycles} determined the uniform densities of tight cycles of length at least 5. The value is $4/27$ when the length is not divisible by 3, and 0 otherwise. Recently, Kielak, Kr\'al', Lamaison, Liu, Shu, and Wu~\cite{kkllsw-tetrahedron} and, independently, Buci\'c~\cite{bucic-tetrahedron} proved that $K_4^{(3)}$ has uniform density $1/2$, answering the corresponding question of Erd\H{o}s and S\'os.

We are concerned with the structure of $\Pu r$. The first general restriction is a gap above zero. Reiher, R\"odl, and Schacht~\cite{rrs-zero} characterized the $3$-graphs of zero uniform Tur\'an density and proved that every positive value is at least $1/27$. Garbe, Kr\'al', and Lamaison~\cite{gkl-minimum} constructed $3$-graphs attaining this lower bound, and Lin, Wang, and Zhou~\cite{lwz-minimum} established the corresponding minimum $r^{-r}$ in every uniformity. Li, Liu, and Pikhurko~\cite[Theorem 1.1]{llp} proved that $[a,1]\subseteq\Pu3$ for some $a\in(0,1)$.

Parallel developments concern the sets of attainable densities in ordinary and minimum $\ell$-degree Tur\'an problems. Erd\H{o}s's theorem on complete multipartite hypergraphs~\cite[Theorem 1]{erdos-multipartite} implies that no ordinary Tur\'an density lies in $(0,r!/r^r)$. Liu and Pikhurko~\cite[Theorem 1.1]{liu-pikhurko-intervals} proved that, for every $r\ge3$, there is a $\delta_r>0$ such that every value in $[1-\delta_r,1]$ occurs as an ordinary Tur\'an density of a possibly infinite forbidden family, answering a question of Frankl, Peng, R\"odl, and Talbot~\cite[Question 6.1]{fprt}. Li, Liu, and Mubayi~\cite[Theorem 1.1]{ordinary} proved that every value in $[r!/r^r,1]$ occurs for all sufficiently large $r$. For minimum degree conditions, Mubayi and Zhao~\cite{mubayi-zhao} proved that codegree Tur\'an densities are dense in $[0,1]$, and Lo and Markstr\"om~\cite{lo-markstrom} extended this to minimum $\ell$-degree densities for every $r>\ell\ge2$. Li and Liu~\cite[Theorem 1.1]{li-liu-degree} proved that every value in $[0,1]$ is realized by a possibly infinite forbidden family whenever $r>\ell\ge2$ and $(r,\ell)\ne(3,2)$.

Lower-bound constructions can be described by palettes, which specify finitely many colors and admissible color patterns. Lamaison~\cite{lamaison} proved that, for every $3$-graph $H$, its uniform Tur\'an density is the supremum of the densities of finite palette constructions avoiding $H$. Thus palette constructions give arbitrarily sharp lower bounds for every forbidden $3$-graph. The theorem also reduces upper bounds to finite colorability questions, without requiring a new hypergraph regularity argument for each forbidden hypergraph. Lin, Sun, Wang, and Zhou~\cite{lswz} extended the characterization to arbitrary forbidden families and all uniformities. A complementary separation theorem uses maps preserving admissible color patterns to characterize when some hypergraph is colorable by one palette but not by another. For $3$-graphs, this theorem was proved independently by King, Piga, Sales, and Sch\"ulke~\cite[Lemma 4.1]{kpss} and by Kr\'al', Ku\v{c}er\'ak, Lamaison, and Tardos~\cite[Theorem 13]{kklt}; Lin, Sun, Wang, and Zhou~\cite[Theorem 1.6]{lswz} extended it to every uniformity.

For a nonempty finite set $C$, let $\mathbb{S}_C\coloneqq\{\bx\in\mathbb R_{\ge0}^{C}\colon\sum_{c\in C}x_c=1\}$ denote the \emph{probability simplex} on $C$. For a positive integer $m$, write $[m]\coloneqq\{1,\ldots,m\}$ and $\mathbb{S}_m\coloneqq\mathbb{S}_{[m]}$.

A \emph{finite $r$-palette} $P$ consists of a nonempty finite color set $C(P)$ and a set of \emph{admissible tuples} in $C(P)^r$. We identify $P$ with its set of admissible tuples, which may be empty, and hence write $|P|$ for the number of admissible tuples of $P$. Its polynomial and \emph{Lagrangian} are
\[
\Lambda_P(\bx)\coloneqq\sum_{(c_1,\ldots,c_r)\in P}\prod_{i=1}^r x_{c_i},
\quad\text{and}\quad
\lambda(P)\coloneqq\max_{\bx\in\mathbb{S}_{C(P)}}\Lambda_P(\bx).
\]
Put $\Du_r\coloneqq\{\lambda(P)\colon P\text{ is a finite }r\text{-palette}\}$. The main theorem determines $\Pu r$ in terms of $\Du_r$.

\begin{theorem}\label{thm:main}
For all sufficiently large $r$, $\Pu r=\bigl(\Du_r\cap[0,4r^{-r})\bigr)\cup[4r^{-r},1]$.
\end{theorem}

\emph{Remark.} GPT 6 pro produced a proof of the ordinary Tur\'an-density theorem for every uniformity at least 18. The proof is recorded in~\cite[Theorem 1.1]{ordinary-quantitative}. Combining this result with the 12-uniform construction and the propagation argument can improve the range of \cref{thm:main} to $r\ge30$.

For the uniformities covered by \cref{thm:main}, every value in $[4r^{-r},1]$ occurs, whereas the finite-attainment theorem below shows that the values below $4r^{-r}$ form a countable set of algebraic numbers. The set $\Pu r\cap[r^{-r},4r^{-r})$ is nevertheless dense in $[r^{-r},4r^{-r})$, as \cref{thm:nonclosed}(i) shows.

A palette is \emph{$\Fu$-deficient} if no member of $\Fu$ can be colored by it; the precise coloring convention is given in \cref{sec:palettes}.

\begin{theorem}\label{thm:attainment}
Let $r\ge3$. The following statements hold.
\begin{enumerate}
\item If $\Fu$ is a family of $r$-graphs with $\pi_{r-2}(\Fu)<4r^{-r}$, then there is a finite $\Fu$-deficient $r$-palette $P$ with $\pi_{r-2}(\Fu)=\lambda(P)$.
\item We have $\Pu r\cap[0,4r^{-r}]=\Pufin r\cap[0,4r^{-r}]=\Du_r\cap[0,4r^{-r}]$, and this common set is countable and consists of algebraic numbers.
\end{enumerate}
\end{theorem}

To obtain the finite-family conclusion, we extend the theorem of King, Piga, Sales, and Sch\"ulke~\cite[Theorem 1.1]{kpss} from $r=3$ to every $r\ge3$ by proving $\Du_r\subseteq\Pufin r$. The proof in \cref{app:finite-family-realization} follows their rigidity and counting argument, with the changes needed for general uniformity and tuples with repeated colors. The realizing family need not be a subfamily of $\Fu$. An equivalent polynomial characterization of the low-density set is given in \cref{sec:palettes}.

To obtain the upper interval in \cref{thm:main}, we first construct an interval with the critical left endpoint in uniformity 12, then lift it to every $r\ge12$.

\begin{theorem}\label{thm:seed}
There exists $\eps>0$ such that, for every $r\ge12$, we have $[4r^{-r},(4+\eps)r^{-r}]\subseteq\Pu r$.
\end{theorem}

Every point of this interval is the limit of the Lagrangians of finite $r$-palettes $P_i$ with homomorphisms $P_{i+1}\to P_i$. These homomorphisms provide the compatibility needed to realize the limiting Lagrangians as uniform densities.

We then use products with finite palettes to extend this endpoint interval to $[4r^{-r},1]$ for all sufficiently large $r$. Approximating the ordinary densities supplied by Li, Liu, and Mubayi~\cite[Theorem 1.1]{ordinary} by finite hypergraph Lagrangians gives the values needed for this propagation; see \cref{sec:interval}.

The next theorem describes the closure of $\Pu r$ and locates its missing limit points.

\begin{theorem}\label{thm:nonclosed}
The following statements hold.
\begin{enumerate}
\item For all sufficiently large $r$, $\overline{\Pu r}=\overline{\Du_r}=\{0\}\cup[r^{-r},1]$, and $\Pu r\cap[r^{-r},4r^{-r})$ is countable and dense in that interval.
\item For every $r\ge3$, the set $\Pu r$ contains no nondegenerate interval contained in $[0,4r^{-r})$.
\item For every $r\ge3$, $\left|\left(\overline{\Pu r}\setminus\Pu r\right)\cap[\frac{54}{25}r^{-r},\frac{9}{4}r^{-r}]\right|=2^{\aleph_0}$. In particular, $\Pu r$ is not closed.
\end{enumerate}
\end{theorem}

In contrast, Pikhurko~\cite[Proposition 1 and Theorem 2]{pikhurko} proved that the set of ordinary Tur\'an densities of possibly infinite forbidden families is closed and has cardinality continuum. We localize his uncountability construction using the non-jump $4/9$ of Liu and Mubayi~\cite{liumubayi} to obtain continuum many limit points of finite-palette Lagrangians in the interval specified in \cref{thm:nonclosed}(iii). All but countably many of these points are transcendental, so finite attainment excludes them from $\Pu r$.

The difficulty in finite attainment is that approximating palettes may have arbitrarily many colors. If $\pi_{r-2}(\Fu)<4r^{-r}$, a finite subfamily of $\Fu$ excludes repeated-color tuples and ladders of a fixed length from every palette deficient for that subfamily. This obstruction gives the compactness needed to construct a countable palette attaining the limiting Lagrangian. We then find a finite subpalette with the same value, retaining the colorability restrictions imposed by the original forbidden family.

The interval construction follows the same basic strategy as the proof of Li, Liu, and Mubayi~\cite{ordinary}. Auxiliary $3$-graphs control the maximizing weights, and an additional term produces overlapping intervals approaching the target endpoint. The main difference is that convergence of finite-palette Lagrangians alone does not guarantee realization as a uniform density. We therefore start with the interval of $3$-uniform palette limits of Li, Liu, and Pikhurko~\cite{llp} and preserve the homomorphisms between successive palettes required by their realization theorem~\cite[Proposition 3.1 and Section 6]{llp}. Whereas Li, Liu, and Mubayi allow small deviations among the products of auxiliary weights appearing in the additional term, the $12$-uniform construction forces the auxiliary colors in each class to have equal weights at every optimum. The resulting maximum depends only on the Lagrangian of the original palette and produces intervals approaching $4\cdot12^{-12}$.

\Cref{sec:palettes} collects the basic facts about palette Lagrangians and colorability. \Cref{sec:compactness} proves finite attainment, and \cref{sec:low-topology} proves nonclosedness and gives low-density examples. \Cref{sec:compiler,sec:interval} construct and propagate intervals of uniform densities to prove \cref{thm:main}. \Cref{sec:concluding} records results for the other uniform density levels. \Cref{app:finite-family-realization} proves finite-family realization in every uniformity and its stability and extremal consequences for palettes. \Cref{app:countable-colorability} gives a countable palette whose Lagrangian is smaller than the uniform density of the hypergraphs it colors. \Cref{sec:products} proves product formulas for uniform densities and a stabilization result for sequences of palettes linked by homomorphisms.

\section{Palettes and Lagrangians}\label{sec:palettes}
By homogeneity, the maximum of $\Lambda_P$ over nonnegative weightings of total weight $r$ is $r^r\lambda(P)$, which we call the \emph{relative Lagrangian} of $P$.
A palette consisting of one tuple of $r$ distinct colors has relative Lagrangian 1, the minimum positive value. For a simple $r$-graph $H$, its edge polynomial is $\sum_{e\in H}\prod_{v\in e}x_v$.

Given a coloring $\chi\colon\binom{V(H)}{r-1}\to C(P)$ of an ordered $r$-graph $H$, we list the face colors of an edge $e=\{v_1<\cdots<v_r\}$ in the order $\bigl(\chi(e\setminus\{v_1\}),\ldots,\chi(e\setminus\{v_r\})\bigr)$.
An ordered $r$-graph is \emph{$P$-colorable} if a coloring of all its $(r-1)$-subsets by $C(P)$ makes every edge admissible in this order. An unordered $r$-graph is $P$-colorable if some vertex order has such a coloring. We write
\[
\Col(P)\coloneqq\{H\colon H\text{ is an unordered }P\text{-colorable }r\text{-graph}\}.
\]

A \emph{subpalette} of $P$ is obtained by deleting colors and admissible tuples. For a nonempty set $J\subseteq C(P)$, the \emph{induced subpalette} on $J$ is $P[J]\coloneqq P\cap J^r$.

We call a map $\theta\colon C(P)\to C(Q)$ a \emph{homomorphism} from $P$ to $Q$, and write $\theta\colon P\to Q$, if $(\theta(c_1),\ldots,\theta(c_r))\in Q$ for every $(c_1,\ldots,c_r)\in P$. The \emph{reverse palette} $\rev(P)$ is defined by $C(\rev(P))\coloneqq C(P)$ and $\rev(P)\coloneqq\{(c_r,\ldots,c_1)\colon(c_1,\ldots,c_r)\in P\}$. If $P\to Q$, then $\Col(P)\subseteq\Col(Q)$ and $\lambda(P)\le\lambda(Q)$. Moreover, we have $\Col(P)=\Col(\rev(P))$ and $\lambda(P)=\lambda(\rev(P))$. For the Lagrangian inequality, take $\bx\in\mathbb{S}_{C(P)}$ and put $y_d\coloneqq\sum_{c\in C(P)\colon\theta(c)=d}x_c$ for each $d\in C(Q)$. Then $\by\in\mathbb{S}_{C(Q)}$ and $\Lambda_P(\bx)\le\Lambda_Q(\by)$, since the expansion of each target monomial contains all source monomials mapped to it.

The following characterization and separation theorem allows us to work with finite palettes.
\begin{theorem}[{\cite[Theorems 1.4 and 1.6]{lswz}}]\label{thm:palette-tools}
The following statements hold.
\begin{enumerate}
\item For any family $\Fu$ of $r$-graphs, $\pi_{r-2}(\Fu)=\sup\{\lambda(P)\colon P\text{ is a finite }\Fu\text{-deficient }r\text{-palette}\}$.
\item For finite $r$-palettes $P,Q$, there is an $r$-graph in $\Col(P)\setminus\Col(Q)$ if and only if neither $P\to Q$ nor $P\to\rev(Q)$.
\end{enumerate}
\end{theorem}
The cited characterization uses the uniform weighting on palette colors. To obtain the weighted formulation, first consider a rational probability weighting $x_c=n_c/N$ and replace each color $c$ by $n_c$ clones, declaring a tuple admissible exactly when its projection belongs to $P$. The resulting palette has density $\Lambda_P(\bx)$ under the uniform weighting and maps to $P$, so it preserves every deficiency condition of $P$. Approximating arbitrary probability weightings by rational ones proves~\cref{thm:palette-tools}(i).

\begin{fact}\label{lem:finite-values}
We have $\Du_r\subseteq\Pu r\subseteq\{0\}\cup[r^{-r},1]$.
\end{fact}
\begin{proof}
For a finite palette $P$, let $\Fu_P$ consist of all finite $r$-graphs outside $\Col(P)$. The palette $P$ is $\Fu_P$-deficient. If $Q$ is also $\Fu_P$-deficient, then $\Col(Q)\subseteq\Col(P)$, since any $Q$-colorable $r$-graph outside $\Col(P)$ would belong to $\Fu_P$. Hence \cref{thm:palette-tools}(ii) gives $Q\to P$ or $Q\to\rev(P)$, so $\lambda(Q)\le\lambda(P)$. \Cref{thm:palette-tools}(i) proves $\pi_{r-2}(\Fu_P)=\lambda(P)$.

To prove the second inclusion, suppose that $P\ne\varnothing$ and choose a tuple whose distinct colors have multiplicities $a_1,\ldots,a_t$. The polynomial $\Lambda_P$ is at least the corresponding monomial. Maximizing that monomial gives
\begin{equation}\label{eq:single-tuple}
\lambda(P)\ge\prod_{i=1}^t\left(\frac{a_i}{r}\right)^{a_i}
=r^{-r}\prod_{i=1}^ta_i^{a_i}\ge r^{-r}.
\end{equation}
Since $P\subseteq C(P)^r$, we have $\Lambda_P(\bx)\le(\sum_{c\in C(P)}x_c)^r=1$ for every $\bx\in\mathbb{S}_{C(P)}$, so $\lambda(P)\le1$. Thus every finite palette has Lagrangian either zero or in $[r^{-r},1]$. By \cref{thm:palette-tools}(i), $\pi_{r-2}(\Fu)$ is the supremum of the Lagrangians of finite $\Fu$-deficient palettes. If this supremum is positive, at least one such palette has positive Lagrangian, and hence $r^{-r}\le\pi_{r-2}(\Fu)\le1$.
\end{proof}

For an $s$-graph $H$, its \emph{Lagrangian} is $\lambda(H)\coloneqq s!\max_{\bx\in\mathbb{S}_{V(H)}}\sum_{e\in H}\prod_{v\in e}x_v$. Write $\Pord s$ for the set of ordinary Tur\'an densities of possibly infinite forbidden families. The normalized Lagrangians $\lambda(H)$ belong to $\Pord s$ and are dense in it, by Brown and Simonovits~\cite{bs}; see also Pikhurko~\cite[Theorem 24]{pikhurko}.

For $r\ge k\ge3$ and a finite $k$-palette $P$, choose distinct new colors $z_1,\ldots,z_{r-k}$. Define $\operatorname{lift}_r(P)$ to have color set $C(\operatorname{lift}_r(P))\coloneqq C(P)\sqcup\{z_1,\ldots,z_{r-k}\}$. Its admissible tuples are
\[
\operatorname{lift}_r(P)\coloneqq\{(c_1,\ldots,c_k,z_1,\ldots,z_{r-k})\colon(c_1,\ldots,c_k)\in P\}.
\]

We will use the following two elementary conversions.
\begin{fact}\label{lem:orientation-lift}
The following statements hold.
\begin{enumerate}[beginpenalty=10000]
\item If $H$ is a finite simple $k$-graph and $1\le m\le k!$, selecting exactly $m$ permutations of each edge gives a $k$-palette $P_m(H)$ with $\lambda(P_m(H))=\frac{m}{k!}\lambda(H)$.
\item If $r\ge k\ge3$ and $P$ is a finite $k$-palette, then $\lambda(\operatorname{lift}_r(P))=\frac{k^k}{r^r}\lambda(P)$. The lifting operation preserves palette homomorphisms.
\end{enumerate}
\end{fact}
\begin{proof}
For (i), the polynomial is $m$ times the edge polynomial of $H$. For (ii), the case $r=k$ is immediate. Otherwise, if the old colors have total weight $u$ and the new colors have weights $x_{z_1},\ldots,x_{z_{r-k}}$, the maximum for these total weights is $\lambda(P)u^k\prod_{i=1}^{r-k}x_{z_i}$, subject to $u+\sum_{i=1}^{r-k}x_{z_i}=1$. AM--GM gives the optimum at $u=k/r$ and $x_{z_i}=1/r$. A homomorphism between the original palettes extends by fixing the new colors.
\end{proof}

A palette is \emph{rainbow} if every admissible tuple has pairwise distinct colors. For a rainbow $r$-palette $P$, let $H(P)$ be its underlying simple $r$-graph, and let $m_e\in\{1,\ldots,r!\}$ count the admissible permutations of each edge $e\in H(P)$. Then
\begin{equation}\label{eq:sym-polynomial}
\Lambda_P(\bx)=\sum_{e\in H(P)}m_e \prod_{v\in e}x_v.
\end{equation}

To describe $\Du_r\cap[0,4r^{-r})$, define
\[
\Lu_{r,\le3}\coloneqq
\left\{
\max_{\bx\in\mathbb{S}_n}\sum_{e\in\binom{[n]}r}m_e\prod_{i\in e}x_i\colon
n\ge1\text{ and }m_e\in\{0,1,2,3\}
\right\}.
\]
For any $n\ge1$ and coefficients $m_e\in\{0,1,2,3\}$, construct a palette $P$ on $[n]$ by selecting exactly $m_e$ distinct permutations of each $e\in\binom{[n]}r$ as admissible tuples. Then $\Lambda_P(\bx)=\sum_{e\in\binom{[n]}r}m_e\prod_{i\in e}x_i$, so the maximum in the definition of $\Lu_{r,\le3}$ is $\lambda(P)$. Hence $\Lu_{r,\le3}\subseteq\Du_r$.

\begin{fact}\label{fact:low-density-multiplicity}
If a finite $r$-palette $P$ satisfies $\lambda(P)<4r^{-r}$, then $P$ is rainbow and every underlying edge has at most 3 admissible permutations. Consequently, $\Du_r\cap[0,4r^{-r})=\Lu_{r,\le3}\cap[0,4r^{-r})$.
\end{fact}
\begin{proof}
If an admissible tuple repeats a color, one of its multiplicities is at least 2, so~\eqref{eq:single-tuple} gives $\lambda(P)\ge4r^{-r}$, a contradiction. Thus $P$ is rainbow. Moreover, assigning weight $1/r$ to each color of an underlying edge $e$ gives $\lambda(P)\ge m_e r^{-r}$, so $m_e<4$ and hence $m_e\le3$. By~\eqref{eq:sym-polynomial}, every finite-palette Lagrangian below $4r^{-r}$ therefore belongs to $\Lu_{r,\le3}$. Together with $\Lu_{r,\le3}\subseteq\Du_r$, this proves the asserted equality.
\end{proof}

A \emph{countable $r$-palette} allows a nonempty countable color set; subpalettes, homomorphisms, and colorability are defined as before. For such a palette $P$, define $\lambda(P)\coloneqq\sup\{\lambda(P[J])\colon\varnothing\ne J\subseteq C(P),\ |J|<\infty\}$. Equivalently, take the supremum of $\Lambda_P(\bx)$ over probability weightings with only finitely many positive coordinates.

For a probability weighting $\bx$ with possibly infinite support, the tuple sum defining $\Lambda_P(\bx)$ converges, since $\sum_{(c_1,\ldots,c_r)\in P}\prod_{i\in[r]}x_{c_i}\le(\sum_{c\in C(P)}x_c)^r=1$. Truncating $\bx$ to increasing finite sets covering $C(P)$, each with positive total weight, and normalizing to total weight 1 gives finitely supported probability weightings whose polynomial values converge to $\Lambda_P(\bx)$. Hence $\Lambda_P(\bx)\le\lambda(P)$, and $\lambda(P)=\sup\{\Lambda_P(\bx)\colon x_c\ge0\text{ for all }c\in C(P),\ \sum_{c\in C(P)}x_c=1\}$, with no restriction on the support of $\bx$.

For countable palettes, this Lagrangian is an analytic extension of the finite-palette notion and need not equal the supremum of the uniform densities attainable by arbitrarily large $P$-colorable hypergraphs. Indeed, \cref{prop:countable-colorability} gives a countable $3$-palette with Lagrangian $2/9$ that colors every finite $3$-graph.

\section{Finite attainment below the threshold}\label{sec:compactness}
We prove \cref{thm:attainment} by excluding ladders of a fixed length from the deficient palettes and then taking a limit of optimal weightings.

\begin{definition}\label{def:ladder}
An \emph{$m$-ladder} in a simple $r$-graph $H$ consists of distinct vertices $b_1,\ldots,b_m$ and $(r-2)$-sets $A_1,\ldots,A_{m-1}$ such that $A_i\cap\{b_i,b_{i+1},\ldots,b_m\}=\varnothing$ for every $i\in[m-1]$ and $A_i\cup\{b_i,b_j\}\in H$ for every $1\le i<j\le m$. The sets $A_i$ may intersect one another and may contain earlier vertices $b_j$ with $j<i$. We call the vertices $b_i$ the \emph{spine vertices} and the sets $A_i$ the \emph{cores}.
A rainbow palette contains an $m$-ladder if its underlying hypergraph does.
\end{definition}

For $r=3$, write $A_i=\{a_i\}$. \Cref{fig:ladder} shows all 10 edges of a $5$-ladder.

\begin{figure}[htbp]
\centering
\begin{tikzpicture}[
  x=1cm,y=1cm,every node/.style={font=\small},
  ladder vertex/.style={circle,fill=black,inner sep=1.8pt},
  ladder core/.style={circle,draw=black,fill=white,inner sep=1.8pt},
  ladder edge/.style={line width=.65pt,line join=round}
]
% Each vertex occurs once. The spine vertices lie on a convex arc so that
% no spine vertex lies on an unrelated triangle boundary.
\foreach \i in {1,2,3,4}{
  \coordinate (a\i) at ({3*(\i-1)},1.6);
}
\foreach \i in {1,2,3,4,5}{
  \coordinate (b\i) at ({3*(\i-1)},{.2125*(\i-1)*(\i-5)});
}
\foreach \i/\j/\edgecolor in {
  1/5/blue!70!black,1/4/blue!70!black,1/3/blue!70!black,1/2/blue!70!black,
  2/5/teal!75!black,2/4/teal!75!black,2/3/teal!75!black,
  3/5/orange!85!black,3/4/orange!85!black,4/5/violet!80!black}{
  \fill[\edgecolor,fill opacity=.09] (a\i)--(b\i)--(b\j)--cycle;
}
\foreach \i/\j/\edgecolor in {
  1/5/blue!70!black,1/4/blue!70!black,1/3/blue!70!black,1/2/blue!70!black,
  2/5/teal!75!black,2/4/teal!75!black,2/3/teal!75!black,
  3/5/orange!85!black,3/4/orange!85!black,4/5/violet!80!black}{
  \draw[ladder edge,\edgecolor] (a\i)--(b\i)--(b\j)--cycle;
}
\foreach \i in {1,2,3,4}{
  \node[ladder core,label=above:{$a_{\i}$}] at (a\i) {};
}
\foreach \i in {1,2,3,4,5}{
  \node[ladder vertex,label=below:{$b_{\i}$}] at (b\i) {};
}
\end{tikzpicture}
\caption{A $5$-ladder for $r=3$. Shaded triangles represent the edges $\{a_i,b_i,b_j\}$ for $i<j$; only marked points are vertices.}
\label{fig:ladder}
\end{figure}
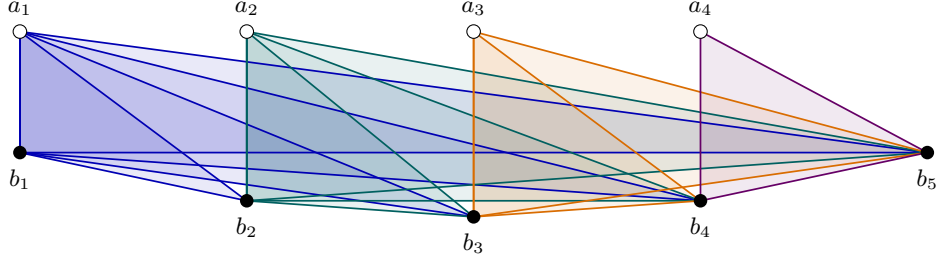

Every pair of spine vertices $b_i,b_j$ is contained in an edge. Consequently, if a rainbow palette containing an $m$-ladder maps to another rainbow palette, its image still contains an $m$-ladder. Indeed, the spine vertices remain distinct, and every required disjointness relation is witnessed inside a mapped edge.

The following compactness theorem is the main step. Its conclusion preserves every deficiency condition shared by the approximating palettes.

\begin{theorem}\label{thm:compactness}
Fix $r\ge3$ and an integer $M\ge2$. Let $P_n$ be finite rainbow $r$-palettes without $M$-ladders. If $\lim_{n\to\infty}\lambda(P_n)=t$, then there are a finite palette $P$ and a subsequence $(P_{n_j})$ such that $\lambda(P)=t$ and $P\to P_{n_j}$ for every sufficiently large $j$.
\end{theorem}

We first obtain the required ladder exclusion from finitely many forbidden hypergraphs. In the compactness proof, we use the Hessian to represent the colors by vectors on a sphere. Ladder exclusion forces the limiting distribution to be concentrated on countably many points of the sphere. We use these points to construct a countable palette attaining the limiting value, then find a finite subpalette with the same Lagrangian.

For distinct $p,q\in[r]$ and $m\ge2$, let $S_m^{p,q}$ have colors $b_1,\ldots,b_m$ and $c_{i,\ell}$ for $i\in[m-1]$ and $\ell\in[r]\setminus\{p,q\}$. For each $i<j$, include the tuple with $b_i$ in position $p$, $b_j$ in position $q$, and $c_{i,\ell}$ in position $\ell$.

For $1\le i<j\le m$, the tuples for $p<q$ and $q<p$ are, respectively,
\[
\begin{gathered}
(c_{i,1},\ldots,c_{i,p-1},b_i,c_{i,p+1},\ldots,c_{i,q-1},b_j,c_{i,q+1},\ldots,c_{i,r}),\\
(c_{i,1},\ldots,c_{i,q-1},b_j,c_{i,q+1},\ldots,c_{i,p-1},b_i,c_{i,p+1},\ldots,c_{i,r}).
\end{gathered}
\]
Empty blocks of $c$-colors are omitted.
Note that $S_m^{p,q}$ is an $m$-ladder palette with a fixed ordering of each edge.

The next lemma uses Ramsey's theorem~\cite[Theorem B]{ramsey} to obtain a homomorphism from some $S_m^{p,q}$ to any rainbow $r$-palette whose underlying $r$-graph contains a sufficiently long ladder.

\begin{lemma}\label{lem:ladder-ramsey}
For every $r\ge3$ and $m\ge2$, there is an integer $M\ge m$ such that every rainbow $r$-palette $P$ whose underlying $r$-graph contains an $M$-ladder admits a homomorphism $S_m^{p,q}\to P$ for some distinct $p,q\in[r]$.
\end{lemma}
\begin{proof}
Let $M\coloneqq R_2(m;r!)$, where $R_2(m;r!)$ is the $r!$-color Ramsey number for pairs. Fix an $M$-ladder in $H(P)$ with spine vertices $b_1,\ldots,b_M$ and cores $A_1,\ldots,A_{M-1}$, and order each core arbitrarily. For each $i<j$, choose an admissible permutation of $A_i\cup\{b_i,b_j\}$. Color the pair $\{i,j\}$ by the positions of $b_i$, $b_j$, and the ordered core vertices in this tuple. There are at most $r!$ colors, so Ramsey's theorem gives a homogeneous set of $m$ indices.

On this set, the earlier and later spine vertices occupy fixed positions $p\ne q$, and each ordered core vertex also occupies a fixed position. Map the spine vertices of $S_m^{p,q}$ to the selected spine vertices in increasing order, and map each $c_{i,\ell}$ to the vertex in the corresponding core occupying position $\ell$. Homogeneity makes this choice independent of the later spine vertex, so the map takes every tuple of $S_m^{p,q}$ to an admissible tuple of $P$. This gives the required homomorphism, which may identify auxiliary colors when different cores intersect.
\end{proof}

To obtain the required ladder exclusion, we compare the palettes $S_m^{p,q}$ with a single repeated-color tuple. Fix distinct positions $p,q\in[r]$, and let $R_{p,q}$ consist of a single tuple with one color $a$ in positions $p,q$ and a different color $z_\ell$ in every other position. For $p<q$, set
\[
R_{p,q}\coloneqq\{(z_1,\ldots,z_{p-1},a,z_{p+1},\ldots,z_{q-1},a,z_{q+1},\ldots,z_r)\}.
\]
Empty blocks of $z$-colors are omitted. We have $R_{p,q}=R_{q,p}$ up to the names of its colors.

\begin{lemma}\label{lem:ladder-coloring}
The following statements hold.
\begin{enumerate}[beginpenalty=10000]
\item We have $\lambda(R_{p,q})=4r^{-r}$.
\item For every $m\ge2$, we have $\lambda(S_m^{p,q})=r^{-r}$.
\item We have $\Col(R_{p,q})=\bigcup_{m=2}^{\infty}\Col(S_m^{p,q})$. More precisely, for every $m\ge2$, each $R_{p,q}$-colorable $r$-graph with at most $m$ vertices is $S_m^{p,q}$-colorable.
\end{enumerate}
\end{lemma}
\begin{proof}
For (i), the polynomial of $R_{p,q}$ is the single monomial $x_a^2\prod_{\ell\in[r]\setminus\{p,q\}}x_{z_\ell}$. Its maximum is attained at $x_a=2/r$ and $x_{z_\ell}=1/r$, giving $\lambda(R_{p,q})=4r^{-r}$.

For (ii), write $x_i$ for the weight of $b_i$, $y_{i,\ell}$ for that of $c_{i,\ell}$, and put $B\coloneqq\sum_{i\in[m]}x_i$ and $Y_\ell\coloneqq\sum_{i\in[m-1]}y_{i,\ell}$. It follows from the definition of $S_m^{p,q}$ that
\[
\Lambda_{S_m^{p,q}}
=\sum_{i\in[m-1]}x_i\left(\sum_{j=i+1}^m x_j\right)\prod_{\ell\in[r]\setminus\{p,q\}}y_{i,\ell}
\le\frac{B^2}{4}\prod_{\ell\in[r]\setminus\{p,q\}}Y_\ell
\le r^{-r}.
\]
For the first inequality, use $x_i\sum_{j=i+1}^m x_j\le B^2/4$ and $\sum_{i\in[m-1]}\prod_{\ell\in[r]\setminus\{p,q\}}y_{i,\ell}\le\prod_{\ell\in[r]\setminus\{p,q\}}Y_\ell$. The last inequality follows by applying AM--GM to $B/2,B/2$ and the $r-2$ numbers $Y_\ell$. Assigning weight $1/r$ to each color of one admissible tuple gives equality.

For (iii), the map $b_i\mapsto a$, $c_{i,\ell}\mapsto z_\ell$ gives a homomorphism $S_m^{p,q}\to R_{p,q}$. To prove the reverse inclusion, fix an $R_{p,q}$-coloring of an $r$-graph with ordered vertices $1,\ldots,N$. The conclusion is immediate for an edgeless $r$-graph, so assume $N\ge r$. For an $(r-1)$-face $T$, write $T[h]$ for its $h$th vertex. Define
\[
h_\ell\coloneqq\begin{cases}q-1,&\text{if }\ell<q,\\q,&\text{if }\ell>q,\end{cases}
\quad\text{and}\quad
\sigma(v)\coloneqq\begin{cases}N+1-v,&\text{if }p<q,\\v,&\text{if }p>q.\end{cases}
\]
Here $h_\ell$ is defined for every $\ell\in[r]\setminus\{q\}$ and belongs to $[r-1]$. Recolor each $(r-1)$-face contained in an edge according to $a\mapsto b_{\sigma(T[h_p])}$ and $z_\ell\mapsto c_{\sigma(T[h_\ell]),\ell}$. Consider an edge $v_1<\cdots<v_r$. The faces in position $p$ and in every position $\ell\ne p,q$ all select $v_q$. They therefore receive the same first index $i\coloneqq \sigma(v_q)$. The face in position $q$ selects $v_{q-1}$ if $p<q$ and $v_{q+1}$ if $p>q$. Its index $j$ satisfies $i<j$ in both cases. The resulting tuple belongs to $S_N^{p,q}$. The color assigned to a face depends only on that face, so the rule is consistent on shared faces. Unused faces can be colored arbitrarily. Finally $S_N^{p,q}$ maps to $S_m^{p,q}$ when $N\le m$.
\end{proof}

We now use \cref{lem:ladder-ramsey,lem:ladder-coloring} to exclude long ladders from deficient palettes.

\begin{lemma}\label{lem:finite-certificate}
Suppose $\pi_{r-2}(\Fu)<4r^{-r}$. There are a finite subfamily $\Fu_0\subseteq\Fu$, with $|\Fu_0|\le\binom r2$, and an integer $M\ge2$ such that every $\Fu_0$-deficient palette is rainbow and has no $M$-ladder in its underlying $r$-graph.
\end{lemma}
\begin{proof}
For each unordered pair $\{p,q\}$, \cref{lem:ladder-coloring}(i) gives $\lambda(R_{p,q})=4r^{-r}$, so $R_{p,q}$ is not $\Fu$-deficient by~\cref{thm:palette-tools}(i). Choose $F_{p,q}\in\Fu\cap\Col(R_{p,q})$, and set
\[
\Fu_0\coloneqq\{F_{p,q}\colon 1\le p<q\le r\}.
\]
Write $F_{q,p}\coloneqq F_{p,q}$ when the positions are listed in the opposite order. Put $m\coloneqq\max\{2,|V(F_{p,q})|\colon p<q\}$ and let $M$ be given by \cref{lem:ladder-ramsey} for these $r$ and $m$.

Let $P$ be an $\Fu_0$-deficient palette. If an admissible tuple of $P$ repeats a color in positions $p,q$, there is a homomorphism $R_{p,q}\to P$; the other colors of $R_{p,q}$ may also be identified. Thus $F_{p,q}$ would be $P$-colorable, contradicting the deficiency of $P$. This excludes all repeated-color tuples.

It remains to exclude an $M$-ladder. If $P$ contains an $M$-ladder, then \cref{lem:ladder-ramsey} gives $S_m^{p,q}\to P$ for some distinct $p,q\in[r]$. By \cref{lem:ladder-coloring}(iii), $F_{p,q}$ is $S_m^{p,q}$-colorable and hence $P$-colorable, a contradiction.
\end{proof}

By \cref{lem:finite-certificate}, if a collection of rainbow palettes has underlying ladders of unbounded length, then any forbidden family for which every palette in the collection is deficient has density at least $4r^{-r}$. This conclusion is independent of the orientations of the underlying edges.

For example, fix distinct positions $p,q$. Let $\Fu$ consist of the $r$-graphs that are not $S_m^{p,q}$-colorable for any $m\ge2$. By \cref{lem:ladder-coloring}(iii), these are exactly the $r$-graphs outside $\Col(R_{p,q})$. The proof of \cref{lem:finite-values} gives $\pi_{r-2}(\Fu)=\lambda(R_{p,q})=4r^{-r}$, whereas $\sup_{m\ge2}\lambda(S_m^{p,q})=r^{-r}$.

\begin{proof}[Proof of \cref{thm:compactness}]
If $t=0$, then $\lambda(P_n)<r^{-r}$ for all sufficiently large $n$. Since every nonempty finite $r$-palette has Lagrangian at least $r^{-r}$ by \cref{lem:finite-values}, we have $P_n=\varnothing$ for all such $n$, and a one-color palette with no admissible tuple suffices. Assume $t>0$ and discard finitely many terms so that $\lambda(P_n)>0$ for every $n$.

To justify the derivative calculations for countable palettes, we first establish convergence and continuity of the defining series.
For a countable palette $P$, the tuple sum defining $\Lambda_P$ extends to $\ell^1(C(P))$, the space of real weightings with $\|\bx\|_1\coloneqq\sum_{c\in C(P)}|x_c|<\infty$. The series converges absolutely because $\sum_{(c_1,\ldots,c_r)\in P}\prod_{i\in[r]}|x_{c_i}|\le\|\bx\|_1^r$. Expanding a difference of products gives
\[
|\Lambda_P(\bx)-\Lambda_P(\by)|
\le r\max\{\|\bx\|_1,\|\by\|_1\}^{r-1}\|\bx-\by\|_1.
\]

If $P$ is rainbow, its first and second derivatives are obtained by differentiating the monomials in~\eqref{eq:sym-polynomial} term by term. For distinct colors $u,v$ and every $\bx\in\ell^1(C(P))$, the resulting series satisfy
\[
\begin{aligned}
|\partial_u\Lambda_P(\bx)|
&\le\sum_{\substack{e\in H(P)\\u\in e}}m_e\prod_{w\in e\setminus\{u\}}|x_w|
\le r\|\bx\|_1^{r-1},\\
|\partial_u\partial_v\Lambda_P(\bx)|
&\le\sum_{\substack{e\in H(P)\\u,v\in e}}m_e\prod_{w\in e\setminus\{u,v\}}|x_w|
\le r(r-1)\|\bx\|_1^{r-2},
\end{aligned}
\]
and $\partial_u^2\Lambda_P=0$. These bounds follow from $m_e\le r!$ and $\sum_{A\in\binom{C(P)}j}\prod_{w\in A}|x_w|\le\|\bx\|_1^j/j!$. Applying the same difference-of-products estimate to the derivative series shows that they depend continuously on $\bx$ in the $\ell^1$ norm, uniformly in their coordinate indices. In particular, the sums defining first and second variations along $\ell^1$ directions converge absolutely.

The real separable \emph{Hilbert spaces} used below can be viewed as finite-dimensional Euclidean spaces or as $\ell^2$, the space of real sequences with $\|\boldsymbol{a}\|^2\coloneqq\sum_{i=1}^{\infty} a_i^2<\infty$ and inner product $\langle\boldsymbol{a},\boldsymbol{b}\rangle\coloneqq\sum_{i=1}^{\infty} a_i b_i$. For finite palettes, finitely many dimensions suffice.

We first represent the entries of the Hessian at an optimal weighting as squared distances between vectors on a sphere. The passage from a quadratic-form inequality to a squared-distance representation follows Schoenberg~\cite{schoenberg}; the argument below also identifies the common radius and the weighted center.

\begin{claim}\label{lem:hessian}
Let $P$ be a finite or countable rainbow $r$-palette. Suppose a probability weighting $\bx$ has all coordinates positive and attains $\lambda(P)>0$. Let $\mathsf{H}\coloneqq(\mathsf{H}_{uv})\in\mathbb{R}^{C(P)\times C(P)}$ be the Hessian of $\Lambda_P$ at $\bx$. There are vectors $\boldsymbol{q}_v$ in a Hilbert space, indexed by $v\in C(P)$, with the following properties.
\begin{enumerate}[beginpenalty=10000]
\item $\|\boldsymbol{q}_v\|^2=r(r-1)\lambda(P)/2$ for every $v\in C(P)$.
\item $\mathsf{H}_{uv}=\|\boldsymbol{q}_u-\boldsymbol{q}_v\|^2$ for every $u,v\in C(P)$.
\item $\sum_{v\in C(P)}x_v\boldsymbol{q}_v=0$.
\item If $P$ is finite and $\bx$ has minimum support among its optimal weightings, then the $\boldsymbol{q}_v$ are pairwise distinct.
\end{enumerate}
\end{claim}
\begin{proof}[Proof of the claim]
Put $d\coloneqq r(r-1)\lambda(P)$. Applying the method of Lagrange multipliers to any two coordinates $u,v$, with all other weights fixed, gives $\partial_u\Lambda_P(\bx)=\partial_v\Lambda_P(\bx)$. Since $\Lambda_P$ is homogeneous of degree $r$, Euler's identity gives $\sum_{v\in C(P)} x_v\partial_v\Lambda_P(\bx)=r\Lambda_P(\bx)=r\lambda(P)$. Since the partial derivatives are equal and $\sum_{v\in C(P)}x_v=1$, the gradient vector satisfies $\nabla\Lambda_P(\bx)=(\partial_v\Lambda_P(\bx))_{v\in C(P)}=r\lambda(P)\one$.

For each color $u$, the first derivative $\partial_u\Lambda_P$ is homogeneous of degree $r-1$. Applying Euler's identity to this function gives $\sum_{v\in C(P)} \mathsf{H}_{uv}x_v=(r-1)\partial_u\Lambda_P(\bx)=r(r-1)\lambda(P)=d$. Hence $\mathsf{H}\bx=d\one$. For every finitely supported $\bz$ with $\sum_{v\in C(P)}z_v=0$, the weighting $\bx+\varepsilon\bz$ is feasible for all sufficiently small positive and negative $\varepsilon$, since $\bx$ is positive on the support of $\bz$. Optimality therefore gives $\left.\frac{d^2}{d\varepsilon^2}\Lambda_P(\bx+\varepsilon\bz)\right|_{\varepsilon=0}=\bz^{\mathsf T}\mathsf{H}\bz\le0$.

All Hessian entries are bounded by $r!/(r-2)!=r(r-1)$. Finitely supported zero-sum vectors are dense in the zero-sum subspace of $\ell^1$, so the second-order inequality extends to every zero-sum $\ell^1$ vector. The derivative identities above are absolutely convergent.

Since the polynomial is squarefree, we have $\mathsf{H}_{vv}=0$. Set $\mathsf{G}\coloneqq(d\one\one^{\mathsf T}-\mathsf{H})/2$. For a finitely supported vector $\bz$, write $s\coloneqq\sum_{v\in C(P)}z_v$ and $\by\coloneqq\bz-s\bx$. The identity $\mathsf{H}\bx=d\one$ gives
\[
\bz^{\mathsf T}\mathsf{G}\bz
=\tfrac12\bigl(ds^2-\bz^{\mathsf T}\mathsf{H}\bz\bigr)
=-\tfrac12\by^{\mathsf T}\mathsf{H}\by\ge0.
\]
Thus $\mathsf{G}$ is a positive semidefinite matrix, also when indexed by a countable set. Its Gram representation gives vectors $\boldsymbol{q}_v$ with $\langle \boldsymbol{q}_u,\boldsymbol{q}_v\rangle=\mathsf{G}_{uv}$. Since $\mathsf{G}_{vv}=d/2$, we have $\|\boldsymbol{q}_v\|^2=d/2$ and $\|\boldsymbol{q}_u-\boldsymbol{q}_v\|^2=\mathsf{G}_{uu}+\mathsf{G}_{vv}-2\mathsf{G}_{uv}=\mathsf{H}_{uv}$, proving (i) and (ii). Moreover, the vector sum $\sum_{v\in C(P)}x_v\boldsymbol{q}_v$ converges absolutely and has squared norm $\bx^{\mathsf T}\mathsf{G}\bx=0$, so it is zero, proving (iii).

For (iv), the weight-transfer argument of Motzkin and Straus~\cite{motzkin-straus} and Frankl and R\"odl~\cite{frankl-rodl} shows that, at an optimal weighting of minimum support on a finite palette, any two colors of positive weight lie in a common underlying edge. Otherwise the polynomial is affine when weight is transferred between those colors, and optimality permits deleting one of them. A common edge, positive weights, and nonnegative coefficients give $\mathsf{H}_{uv}>0$. The distance formula then implies $\boldsymbol{q}_u\ne \boldsymbol{q}_v$.
\end{proof}

For a Borel probability measure $\nu$ on a separable Hilbert space, a point $\boldsymbol{q}$ is an \emph{atom} if $\nu(\{\boldsymbol{q}\})>0$. The set of atoms is at most countable. We call $\nu$ \emph{atomic}, or \emph{purely atomic}, if all its probability lies on these atoms, that is, $\sum_{\boldsymbol{q}\colon\nu(\{\boldsymbol{q}\})>0}\nu(\{\boldsymbol{q}\})=1$. The restriction of $\nu$ to the complement of its atoms is its \emph{nonatomic part}.

To apply \cref{lem:hessian} to each $P_n$, choose an optimal weighting of minimum support and restrict $P_n$ to this support. Denote its positive weights by $x_v^{(n)}$. Put $c_n\coloneqq r(r-1)\lambda(P_n)$ and $c\coloneqq r(r-1)t$, so $\lim_{n\to\infty}c_n=c$. By \cref{lem:hessian}, there are pairwise distinct vectors $\boldsymbol{p}_v^{(n)}$ with $\|\boldsymbol{p}_v^{(n)}\|^2=c_n/2$ and $\mathsf{H}_{uv}^{(n)}=c_n-2\langle \boldsymbol{p}_u^{(n)},\boldsymbol{p}_v^{(n)}\rangle$. Let $\mu_n\coloneqq\sum_{v\in C(P_n)}x_v^{(n)}\delta_{\boldsymbol{p}_v^{(n)}}$ be the probability distribution of these vectors. Here $\delta_{\boldsymbol{q}}$ denotes the probability measure concentrated at $\boldsymbol{q}$.

To control uniformly in $n$ the error from truncating inner products to finitely many coordinates, rotate each configuration so that its weighted second-moment matrix is diagonal. Its $(i,j)$ entry is $\sum_{v\in C(P_n)} x_v^{(n)}p_{v,i}^{(n)}p_{v,j}^{(n)}$. The matrix is real symmetric and positive semidefinite, so the spectral theorem provides an orthogonal change of coordinates that diagonalizes it and preserves all norms and inner products. Order the diagonal entries so that $\rho_{n,1}\ge\rho_{n,2}\ge\cdots\ge0$; their sum is $c_n/2$. Embed all these configurations into $\ell^2$. They lie in a common closed ball centered at the origin. In the coordinatewise topology, this ball is a closed subset of a countable product of compact intervals, since its radius bound can be checked on finite sums of squared coordinates. The ball is therefore compact and metrizable. Passing to a subsequence, $\mu_n$ converges weakly on that ball to a probability measure $\mu$, meaning that $\lim_{n\to\infty}\int\varphi\dd\mu_n=\int\varphi\dd\mu$ for every bounded continuous real-valued function $\varphi$ on the ball with its coordinatewise topology. The Skorokhod representation theorem~\cite{billingsley} gives random vectors $\boldsymbol{p}_n$ and $\boldsymbol{p}$, with distributions $\mu_n$ and $\mu$, on $([0,1],\mathrm{Leb})$ such that, for almost every $s$, we have $\lim_{n\to\infty}p_{n,j}(s)=p_j(s)$ for every $j\ge1$. Here $\mathrm{Leb}$ denotes Lebesgue measure, and for measurable subsets of $[0,1]^k$ we write $|\cdot|$ for their Lebesgue measure. For almost every $s$ and every positive integer $L$, coordinatewise convergence gives
\[
\sum_{j=1}^L p_j(s)^2
=\lim_{n\to\infty}\sum_{j=1}^L p_{n,j}(s)^2
\le\lim_{n\to\infty}\frac{c_n}{2}
=\frac c2.
\]
Letting $L$ tend to infinity shows that $\|\boldsymbol{p}(s)\|^2\le c/2$ almost everywhere. Since the finite vectors are pairwise distinct, $\boldsymbol{p}_n(s)$ determines a unique color $\phi_n(s)$ of $P_n$, with distribution $(x_v^{(n)})_v$.

We next show that if $\mu$ has a nonzero nonatomic part or an atom of squared norm less than $c/2$, then one of the finite palettes $P_n$ contains an $M$-ladder. For this purpose, we express the Hessian entries as symmetric kernels and prove their $L^2$ convergence. Define
\[
K_n(s,t)\coloneqq c_n-2\langle \boldsymbol{p}_n(s),\boldsymbol{p}_n(t)\rangle,
\quad\text{and}\quad
K(s,t)\coloneqq c-2\langle \boldsymbol{p}(s),\boldsymbol{p}(t)\rangle.
\]
\begin{claim}\label{clm:kernel-convergence}
The kernels satisfy $\lim_{n\to\infty}\|K_n-K\|_{L^2([0,1]^2)}=0$. That is, $\lim_{n\to\infty}\int_{[0,1]^2}|K_n(s,t)-K(s,t)|^2\dd s\dd t=0$.
\end{claim}
\begin{proof}[Proof of the claim]
Put $C_0\coloneqq\sup_{n\ge1}c_n/2<\infty$. For every $n\ge1$ and positive integer $L$, the diagonal entries $\rho_{n,j}$ of the second-moment matrix are nonnegative and nonincreasing, with sum $c_n/2\le C_0$, so $\rho_{n,L+1}\le C_0/(L+1)$. It follows that $\sum_{j=L+1}^{\infty}\rho_{n,j}^2\le\rho_{n,L+1}\sum_{j=L+1}^{\infty}\rho_{n,j}\le C_0^2/(L+1)$.

We now apply this estimate to the error in truncating the inner product. By the diagonalization, $\int_0^1 p_{n,i}(s)p_{n,j}(s)\dd s=0$ for $i\ne j$, and $\int_0^1 p_{n,j}(s)^2\dd s=\rho_{n,j}$. Expanding the square and integrating therefore gives
\[
\iint_{[0,1]^2}\left|\langle \boldsymbol{p}_n(s),\boldsymbol{p}_n(t)\rangle-
\sum_{j=1}^L p_{n,j}(s)p_{n,j}(t)\right|^2\dd s\dd t
=\sum_{j=L+1}^{\infty}\rho_{n,j}^2
\le\frac{C_0^2}{L+1}.
\]
For each fixed $L$, the truncated sums converge in $L^2$ by the almost-everywhere coordinatewise convergence of $\boldsymbol{p}_n$ to $\boldsymbol{p}$ and bounded convergence. To obtain the corresponding tail bound for $\boldsymbol{p}$, note that, for each $j$, bounded convergence gives $\rho_j\coloneqq \int_0^1 p_j(s)^2\dd s=\lim_{n\to\infty}\rho_{n,j}$, and, for $i\ne j$, $\int_0^1 p_i(s)p_j(s)\dd s=\lim_{n\to\infty}\int_0^1 p_{n,i}(s)p_{n,j}(s)\dd s=0$. The sequence $(\rho_j)$ is nonincreasing. Moreover, for every positive integer $L$, $\sum_{j=1}^L\rho_j=\lim_{n\to\infty}\sum_{j=1}^L\rho_{n,j}\le C_0$, so $\sum_{j=1}^{\infty}\rho_j\le C_0$. Thus $\sum_{j=L+1}^{\infty}\rho_j^2\le C_0^2/(L+1)$, so the same tail estimate holds for $\boldsymbol{p}$. First choose $L$ large, then let $n$ tend to infinity, which proves the claim.
\end{proof}

In particular, convergence holds in $L^1$. The kernels $K_n$ are nonnegative, being Hessian entries, and $K$ is nonnegative because $\|\boldsymbol{p}(s)\|^2\le c/2$.

Let $m_e^{(n)}$ count the admissible permutations of $e\in H(P_n)$. Define a symmetric function $\omega_n\colon[0,1]^r\to[0,r!]$ by
\[
\omega_n(s_1,\ldots,s_r)\coloneqq
\begin{cases}
m_e^{(n)},&\text{if }e=\{\phi_n(s_1),\ldots,\phi_n(s_r)\}\in H(P_n),\\
0,&\text{otherwise}.
\end{cases}
\]
Then $0\le \omega_n\le r!$ and
\begin{equation}\label{eq:f-hessian}
\int_{[0,1]^{r-2}}\omega_n(s,t,u_3,\ldots,u_r)\dd u_3\cdots\dd u_r
=(r-2)!K_n(s,t).
\end{equation}
The factor $(r-2)!$ counts the orders of the remaining colors of an edge after the first two colors have been fixed; each Hessian monomial occurs once for every such order.

\begin{claim}\label{clm:atomic-limit}
The measure $\mu$ is purely atomic, and every atom has squared norm $c/2$.
\end{claim}
\begin{proof}[Proof of the claim]
Let $U\subseteq[0,1]$ be the set of points for which $\boldsymbol{p}(s)$ is not an atom of $\mu$ having squared norm $c/2$. We have $K(s,t)>0$ almost everywhere on $U\times U$. Indeed, $K(s,t)=0$ requires $\boldsymbol{p}(s)=\boldsymbol{p}(t)$ and both vectors to have squared norm $c/2$. The event $K(s,t)=0$ has positive product measure only at an atom with that squared norm, and all such atoms were removed in defining $U$.

Suppose $|U|>0$, and fix $a\in(0,|U|]$. Choose $\gamma>0$ such that $|\{(s,t)\in U^2\colon K(s,t)<\gamma\}|<a^2/2$. For every measurable $S\subseteq U$ with $|S|\ge a$, at least half of $S^2$ lies where $K\ge\gamma$. By \cref{clm:kernel-convergence}, for all sufficiently large $n$,
\[
\int_{S\times S}K_n
\ge\frac{\gamma a^2}{2}-\|K_n-K\|_{L^1}
\ge\frac{\gamma a^2}{4}.
\]
This bound holds uniformly for all such $S$, including sets chosen after $n$.

Since $\omega_n\le r!$, equation~\eqref{eq:f-hessian} gives
\[
\int_{S\times S\times[0,1]^{r-2}}\ind{\omega_n>0}
\ge\frac1{r(r-1)}\int_{S\times S}K_n.
\]
If $|S|\ge a$, averaging over $(t,u_3,\ldots,u_r)\in S\times[0,1]^{r-2}$ gives a choice with $\int_S\ind{\omega_n(s,t,u_3,\ldots,u_r)>0}\dd s\ge\gamma a^2/(4r(r-1))$, since the averaging domain has measure at most 1. Put $b\coloneqq\phi_n(t)$ and $A\coloneqq\{\phi_n(u_3),\ldots,\phi_n(u_r)\}$. Then $S'\coloneqq\{s\in S\colon A\cup\{b,\phi_n(s)\}\in H(P_n)\}$ has measure at least $\gamma a^2/(4r(r-1))$. Positivity ensures that $A$ has $r-2$ distinct colors and is disjoint from $b$.

Starting with $a_0\coloneqq|U|$, make the following choices successively for $0\le i<M-1$. Choose $\gamma_i>0$ as above with $a=a_i$, and then choose $0<a_{i+1}\le\min\{a_i/2,\gamma_i a_i^2/(8r(r-1))\}$. After these finitely many choices, choose a single $n$ for which $\int_{S\times S}K_n\ge\gamma_i a_i^2/4$ holds for every $0\le i<M-1$ and every measurable $S\subseteq U$ with $|S|\ge a_i$. Starting with $S_0\coloneqq U$, repeat the preceding neighborhood selection $M-1$ times to obtain nested sets $S_i$ of measure at least $a_i$, spine vertices $b_i$, and cores $A_i$. Choose $s\in S_{M-1}$ and put $b_M\coloneqq\phi_n(s)$. Every later spine vertex is the color of a point in each earlier selected neighborhood. By the definitions of these neighborhoods, it is distinct from the corresponding earlier spine vertex and from every color in its core. Thus, for every $1\le i<j\le M$, we have $A_i\cap\{b_i,\ldots,b_M\}=\varnothing$ and $A_i\cup\{b_i,b_j\}\in H(P_n)$. This is an $M$-ladder, a contradiction.

Thus $|U|=0$, which proves the claim.
\end{proof}

By \cref{clm:atomic-limit}, write the limit distribution as $\mu=\sum_{i\in I}w_i\delta_{\boldsymbol{p}_i}$, where $I$ is finite or countably infinite, the $\boldsymbol{p}_i$ are distinct with $\|\boldsymbol{p}_i\|^2=c/2$, and the weights satisfy $w_i>0$ and $\sum_{i\in I}w_i=1$.

We now use the atoms as colors while retaining the limiting Lagrangian. Sampling a point from the preimage of each atom under $\boldsymbol{p}$ and taking its color in $P_n$ gives palettes with homomorphisms into $P_n$, and passing to a limit will recover the value $t$.

Let $B_i\coloneqq\{s\colon \boldsymbol{p}(s)=\boldsymbol{p}_i\}$, so $|B_i|=w_i$. Independently for $i\in I$, choose $\xi_i$ according to normalized Lebesgue measure on $B_i$. Define a random palette $Q_n$ on $I$ by declaring $(i_1,\ldots,i_r)$ admissible exactly when $(\phi_n(\xi_{i_1}),\ldots,\phi_n(\xi_{i_r}))\in P_n$. It is therefore rainbow. For every finite $J\subseteq I$, the map $i\mapsto\phi_n(\xi_i)$ is a homomorphism $Q_n[J]\to P_n$. By the observation after \cref{def:ladder}, homomorphisms between rainbow palettes preserve ladders. Thus an $M$-ladder in $Q_n$ would give one in $P_n$, so every $Q_n$ has no $M$-ladder.

Identify a palette on $I$ with the $0/1$ indicators of its admissible tuples, indexed by $I^r$. The resulting space $\{0,1\}^{I^r}$ is compact and metrizable in the product topology. A sequence converges in this topology precisely when the indicator of each fixed tuple eventually agrees with its indicator in the limit. Passing to a subsequence, the laws of $Q_n$ converge weakly to the law of a random palette $Q$.

\begin{claim}\label{clm:sampled-optimum}
Almost surely, $Q$ is rainbow, has no $M$-ladder, and satisfies $\Lambda_Q(\bw)=\lambda(Q)=t$ for $\bw\coloneqq(w_i)_{i\in I}$.
\end{claim}
\begin{proof}[Proof of the claim]
Any $M$-ladder is witnessed by finitely many admissible tuples, so it persists under sufficiently small changes in the product topology. Thus the set of palettes containing no $M$-ladder is closed. The rainbow condition is also closed, so weak convergence implies that $Q$ is almost surely rainbow with no $M$-ladder. Fix a nonempty finite $J\subseteq I$ and a palette $R$ on $J$ with $\lambda(R)>t$. If $Q_n[J]=R$, the sampling map gives $R\to P_n$, and hence $\lambda(R)\le\lambda(P_n)$. Since $\lim_{n\to\infty}\lambda(P_n)=t$, this event is impossible for all sufficiently large $n$. It depends only on the coordinates indexed by $J^r$ and is therefore both open and closed, so weak convergence gives $\Prb[Q[J]=R]=\lim_{n\to\infty}\Prb[Q_n[J]=R]=0$. There are countably many choices of $J$ and $R$, so every finite subpalette of $Q$ has Lagrangian at most $t$ almost surely. By the countable-palette convention, $\lambda(Q)\le t$ almost surely.

We show that the fixed weighting $\bw=(w_i)_{i\in I}$ attains $t$ in $Q$. Choose $s_1,\ldots,s_r$ independently and uniformly from $[0,1]$. Since each $\phi_n(s_j)$ has distribution $(x_v^{(n)})_v$, the probability that $(\phi_n(s_1),\ldots,\phi_n(s_r))\in P_n$ is $\lambda(P_n)$. The expectation of $\Lambda_{Q_n}(\bw)$ is
\[
\E[\Lambda_{Q_n}(\bw)]=\Prb\!\left[(\phi_n(s_1),\ldots,\phi_n(s_r))\in P_n\text{ and }\boldsymbol{p}(s_i)\ne\boldsymbol{p}(s_j)\text{ for every }1\le i<j\le r\right].
\]
The second condition says that the points lie in pairwise different sets $B_i$. Indeed, for distinct indices $i_1,\ldots,i_r$, the event $s_j\in B_{i_j}$ for every $j\in[r]$ has probability $w_{i_1}\cdots w_{i_r}$; conditional on this event, the points have the same joint law as $\xi_{i_1},\ldots,\xi_{i_r}$. Terms with repeated indices contribute zero because $Q_n$ is rainbow. Thus $\lambda(P_n)-\E[\Lambda_{Q_n}(\bw)]$ is the probability that the color tuple is admissible and at least 2 sampled points lie in the same $B_i$.

Since $K=0$ on each $B_i\times B_i$ and these sets are pairwise disjoint, we have
\[
\sum_{i\in I}\int_{B_i\times B_i}K_n\le\|K_n-K\|_{L^1}.
\]
By \cref{clm:kernel-convergence}, $\lim_{n\to\infty}\|K_n-K\|_{L^1}=0$. Whenever the color tuple is admissible, $\omega_n(s_1,\ldots,s_r)\ge1$. A union bound over pairs of sampled points, followed by the symmetry of $\omega_n$ and~\eqref{eq:f-hessian}, gives
\[
0\le\lambda(P_n)-\E[\Lambda_{Q_n}(\bw)]
\le\binom r2(r-2)!\sum_{i\in I}\int_{B_i\times B_i}K_n.
\]
It follows that $\lim_{n\to\infty}\E[\Lambda_{Q_n}(\bw)]=t$. To pass to the limiting law, observe that, for fixed $\bw$, the value $\Lambda_Q(\bw)$ depends continuously on the admissible-tuple indicators in the topology of coordinatewise convergence. Indeed, for a nonempty finite set $J\subseteq I$, deleting the tuples that meet $I\setminus J$ while keeping the remaining weights unchanged gives
\[
0\le\Lambda_Q(\bw)-\Lambda_{Q[J]}((w_i)_{i\in J})
\le1-\left(\sum_{i\in J}w_i\right)^r
\le r\sum_{i\in I\setminus J}w_i.
\]
The bound is independent of $Q$. As $J$ increases to $I$, the resulting polynomials depend on finitely many tuple indicators and converge uniformly to $\Lambda_Q(\bw)$, proving continuity. The preceding loss estimate and convergence of laws give $\E[\Lambda_Q(\bw)]=t$. Since $\Lambda_Q(\bw)\le\lambda(Q)\le t$ almost surely, equality holds almost surely throughout, as required.
\end{proof}

It remains to show that a finite subpalette attains the same Lagrangian as the countable palette.

\begin{claim}\label{lem:countable-compression}
Let $P$ be a countable rainbow palette with no $M$-ladder for some integer $M\ge2$. If a probability weighting $\bx$ attains a positive value $\lambda(P)$, then a finite subpalette of $P$ has Lagrangian $\lambda(P)$.
\end{claim}
\begin{proof}[Proof of the claim]
Discard zero-weight colors and apply \cref{lem:hessian} to $P$ with its optimal weighting $\bx$, writing $d\coloneqq r(r-1)\lambda(P)$ and $\boldsymbol{q}_v$ for the resulting vectors. Their common norm is $\sqrt{d/2}$. Suppose first that the vectors $\boldsymbol{q}_v$ take infinitely many distinct values. Choose colors $b_1,b_2,\ldots$ with pairwise distinct vectors. Their boundedness allows passage to a subsequence along which $\boldsymbol{q}_{b_n}$ converges weakly to $\boldsymbol{q}_*$, meaning that $\lim_{n\to\infty}\langle\boldsymbol{q}_{b_n},\bz\rangle=\langle\boldsymbol{q}_*,\bz\rangle$ for every vector $\bz$ in the Hilbert space. For each $(r-1)$-set $K$ of colors, set $m_{K\cup\{b_n\}}\coloneqq0$ when $b_n\in K$. There are only countably many such sets $K$, and each coefficient takes values in $\{0,1,\ldots,r!\}$. A diagonal argument therefore gives a further subsequence along which, for every fixed $K$, the coefficient $m_{K\cup\{b_n\}}$ is constant for all sufficiently large $n$.

For a fixed color $v$, we have $\lim_{n\to\infty}\mathsf{H}_{b_nv}=d-2\langle \boldsymbol{q}_*,\boldsymbol{q}_v\rangle$.
The right side is nonnegative. It is zero only if $\|\boldsymbol{q}_*\|=\sqrt{d/2}$ and $\boldsymbol{q}_v=\boldsymbol{q}_*$. Hence it is positive for all but at most one of the selected colors $b_i$.

Expanding the Hessian entry gives
\[
\mathsf{H}_{b_nv}=\sum_{\substack{A\subseteq C(P)\setminus\{v,b_n\}\\|A|=r-2}}
 m_{A\cup\{v,b_n\}}\prod_{w\in A}x_w.
\]
For all sufficiently large $n$, we have $b_n\ne v$. Using the zero-coefficient convention above, we may take the sum over the fixed family of $(r-2)$-sets $A\subseteq C(P)\setminus\{v\}$. For each such $A$, the preceding diagonal choice with $K=A\cup\{v\}$ makes $m_{A\cup\{v,b_n\}}$ eventually constant. Since $0\le m_{A\cup\{v,b_n\}}\le r!$ and $\sum_{A\in\binom{C(P)}{r-2}}\prod_{w\in A}x_w\le1/(r-2)!$, the dominated convergence theorem for series gives
\[
\lim_{n\to\infty}\mathsf{H}_{b_nv}
=\sum_{\substack{A\subseteq C(P)\setminus\{v\}\\|A|=r-2}}
\left(\lim_{n\to\infty}m_{A\cup\{v,b_n\}}\right)\prod_{w\in A}x_w.
\]
If $\lim_{n\to\infty}\mathsf{H}_{b_nv}>0$, at least one summand on the right is positive. Its core $A$ therefore satisfies $\lim_{n\to\infty}m_{A\cup\{v,b_n\}}\ge1$. Since this coefficient is eventually constant, we have $A\cup\{v,b_n\}\in H(P)$ for every sufficiently large $n$.

We now select indices $n_1<\cdots<n_M$, avoiding the possible single color with zero Hessian limit. After choosing $b_{n_i}$ for $i<M$, choose a core $A_i$ by the preceding argument with $v=b_{n_i}$. Then choose $n_{i+1}>n_i$ so large that $b_{n_{i+1}}\notin\bigcup_{j=1}^i A_j$ and $A_j\cup\{b_{n_j},b_{n_{i+1}}\}\in H(P)$ for every $j\in[i]$, again avoiding the exceptional color. These choices give $A_i\cap\{b_{n_i},\ldots,b_{n_M}\}=\varnothing$ for every $i\in[M-1]$, and every $A_i\cup\{b_{n_i},b_{n_j}\}$ with $i<j$ is an edge. Thus the selected spine vertices and cores form an $M$-ladder, a contradiction.

The vectors $\boldsymbol{q}_v$ thus take only finitely many distinct values. Partition the colors according to those values. If $u,v$ belong to one part, then $\mathsf{H}_{uv}=0$. Since all weights are positive and all coefficients are nonnegative, no underlying edge contains both colors. The polynomial is therefore affine in the weights of any one part when the weights in the other parts are fixed. At the current optimal weighting, the coefficients of $x_v$ in this affine function are the partial derivatives $\partial_v\Lambda_P(\bx)$, which are equal for all colors $v$ in that part by the two-coordinate first-order condition. Concentrating the total weight of that part on one color leaves the value unchanged. The resulting weighting is still optimal, so the same argument applies successively to all the other parts. Since there are only finitely many parts, we obtain a finitely supported probability weighting attaining $\lambda(P)$.
\end{proof}

By \cref{clm:sampled-optimum,lem:countable-compression}, a finite subpalette of $Q$ has Lagrangian $t$ almost surely.

To obtain a fixed finite palette, note that there are countably many finite subsets $J\subseteq I$ and only finitely many palettes on each. Hence some fixed finite $J$ and some fixed palette $P$ on $J$ satisfy $\Prb[Q[J]=P]>0$ and $\lambda(P)=t$. The event $Q[J]=P$ depends only on the finitely many coordinates indexed by $J^r$, so it is both open and closed in the product topology. Weak convergence therefore gives $\lim_{n\to\infty}\Prb[Q_n[J]=P]=\Prb[Q[J]=P]>0$. For every sufficiently large $n$ along the chosen subsequence there is therefore a realization with $Q_n[J]=P$. The homomorphism $i\mapsto\phi_n(\xi_i)$ then gives $P\to P_n$. Composing with the inclusions into the original palettes, before restriction to the supports of the chosen optimal weightings, proves the statement of the theorem.
\end{proof}

Note that the finite palette in \cref{thm:compactness} may depend on the whole convergent sequence. The proof gives no uniform bound on optimal support sizes in terms of $r$ and $M$.

\begin{proof}[Proof of \cref{thm:attainment}]
For (i), if $\pi_{r-2}(\Fu)=0$, use a palette with no admissible tuples. Otherwise, by~\cref{thm:palette-tools}(i), choose finite $\Fu$-deficient palettes $P_n$ with $\lim_{n\to\infty}\lambda(P_n)=\pi_{r-2}(\Fu)$. By \cref{lem:finite-certificate}, the palettes $P_n$ are rainbow and contain no $M$-ladder for some $M$ independent of $n$. \Cref{thm:compactness} gives a finite palette $P$ with $\lambda(P)=\pi_{r-2}(\Fu)$ and $P\to P_n$ along a subsequence. If a member of $\Fu$ were $P$-colorable, it would be $P_n$-colorable, a contradiction. Thus $P$ is $\Fu$-deficient.

For (ii), \cref{lem:finite-values} gives $\Du_r\subseteq\Pu r$. The polynomial description of the densities below $4r^{-r}$ follows from the rainbow-palette discussion in \cref{sec:palettes}.

To prove algebraicity, we use the fact that the maximum of a rational-coefficient polynomial $f$ over $\mathbb{S}_m$ is a real algebraic number. Its value $y$ is specified by
\[
\exists\bx\in\mathbb{S}_m\ \bigl(f(\bx)=y\ \wedge\ \forall\bz\in\mathbb{S}_m\ f(\bz)\le y\bigr).
\]
Quantifier elimination over real closed fields~\cite{bpr} describes this singleton by finitely many polynomial sign conditions over $\Q$. At least one nonzero polynomial in this description must vanish at $y$; otherwise all its sign conditions would remain unchanged in a neighborhood of $y$, contradicting that the set is a singleton. Thus $y$ is algebraic. Countability follows because there are countably many finite palettes.

By \cref{thm:app-fr-realization}, $\Du_r\subseteq\Pufin r\subseteq\Pu r$. Finite attainment makes these inclusions equalities below $4r^{-r}$. The palette with the single tuple $(a,a,c_3,\ldots,c_r)$, where $a,c_3,\ldots,c_r$ are pairwise distinct, has Lagrangian $(2/r)^2(1/r)^{r-2}=4r^{-r}$. The same realization theorem puts this endpoint in all three sets.
\end{proof}

\section{Nonclosedness}\label{sec:low-topology}
We prove \cref{thm:nonclosed}(iii) by localizing Pikhurko's ordinary-density construction and applying finite attainment to exclude the resulting transcendental limit points. We then give two examples of finite-palette values below $4/27$.

\begin{proof}[Proof of \cref{thm:nonclosed}(iii)]
We first locate continuum many ordinary Tur\'an densities in a short interval. A number $\alpha\in[0,1)$ is an \emph{ordinary non-jump} for $3$-graphs if $(\alpha,\alpha+\eps)\cap\Pord3\ne\varnothing$ for every $\eps>0$. Frankl and R\"odl~\cite{frankl-rodl} proved that ordinary non-jumps exist in every uniformity at least 3. For $m\ge3$, we have $\lambda(K_m^{(3)})=(m-1)(m-2)/m^2$.

\begin{claim}\label{lem:localized-ordinary}
Let $m\ge3$. If $\alpha$ is an ordinary non-jump and $\alpha<(m-1)(m-2)/m^2$, then $|\Pord3\cap[(m-1)(m-2)/m^2,(m-2)/(m+1)]|=2^{\aleph_0}$.
\end{claim}
\begin{proof}[Proof of the claim]
Pikhurko~\cite[Section 6, proof of Theorem 2]{pikhurko} constructs distinct ordinary densities from infinite subsets of $\N$. His proof permits any ordinary non-jump $\alpha$ and any $m$ with $\alpha<(m-1)(m-2)/m^2$. It selects finite $3$-graphs $H_i$ whose normalized Lagrangians lie in $(\alpha,(m-1)(m-2)/m^2)$ and decrease sufficiently rapidly to $\alpha$. Each infinite set $A=\{a_1<a_2<\cdots\}\subseteq\N$ specifies a recursive construction on an $m$-part partition. Add every triple meeting 3 distinct classes, place a blowup of $H_{a_1}$ in the second class, and repeat the construction with $A\setminus\{a_1\}$ in the first class. Condition (29) in that proof specifies the required rate of decrease, and the comparison argument following (30) shows that the resulting ordinary densities $\tau_A$ are pairwise distinct. The transversal triples give $\tau_A\ge(m-1)(m-2)/m^2$.

We prove the additional upper bound $\tau_A\le (m-2)/(m+1)$. For $\bx\in\mathbb{S}_m$, let $E_3\coloneqq\sum_{1\le i<j<k\le m}x_ix_jx_k$ and $Q_3\coloneqq\sum_{\substack{i,j\in[m]\\i\ne j}}x_i^2x_j$. For each triple of indices, applying AM--GM to its 6 terms of the form $x_i^2x_j$ gives a sum at least $6x_ix_jx_k$. Summing over triples gives $(m-2)Q_3\ge6E_3$. Since $1-\sum_{i\in[m]} x_i^3=3Q_3+6E_3$, we obtain $6E_3\le \frac{m-2}{m+1}\left(1-\sum_{i\in[m]} x_i^3\right)$. We also have $(m-1)(m-2)/m^2<(m-2)/(m+1)$. Suppose the normalized Lagrangian of the recursive part is at most $(m-2)/(m+1)$. If $x_i$ is the total weight in class $i$, the normalized polynomial of the whole construction is at most
\[
6E_3+\frac{m-2}{m+1}x_1^3+\lambda(H_{a_1})x_2^3
\le6E_3+\frac{m-2}{m+1}\sum_{i\in[m]} x_i^3\le\frac{m-2}{m+1}.
\]
Starting with an edgeless $3$-graph at the final stage of the recursion, induction gives this bound for every finite stage. In particular, every finite $3$-graph $G$ on $N$ vertices obtained from this construction satisfies $6|G|/N^3\le (m-2)/(m+1)$. Pikhurko identifies $\tau_A$ with the limit of the maximum edge densities of these $3$-graphs as $N$ tends to infinity. Since $\lim_{N\to\infty}N^3/(6\binom N3)=1$, we obtain $\tau_A\le(m-2)/(m+1)$. The pairwise distinct values $\tau_A$ prove the claim.
\end{proof}

Liu and Mubayi~\cite[Theorem 1.1]{liumubayi} proved that $4/9$ is an ordinary non-jump. Since $4/9<12/25$, applying \cref{lem:localized-ordinary} with $m=5$ gives $|\Pord3\cap[12/25,1/2]|=2^{\aleph_0}$.

The set $\mathcal S\coloneqq\frac16(\Pord3\cap[12/25,1/2])\subseteq[2/25,1/12]$ has cardinality continuum. By ordinary Lagrangian approximation and \cref{lem:orientation-lift}, each point of $\mathcal S$ is the limit of a sequence in $\Du_3$. Since $1/12<4/27$, the sequence eventually lies in $\Du_3\cap[0,4/27]$. Thus $\mathcal S\subseteq\overline{\Du_3\cap[0,4/27]}$. Deleting the countably many algebraic numbers leaves continuum many transcendental points. Every element of $\Du_3\cap[0,4/27]$ is algebraic, so none of these points belongs to $\Du_3\cap[0,4/27]$.

We now lift these limit points to uniformity $r$ and use finite attainment to exclude them from $\Pu r$. Fix a transcendental point $t\in\mathcal S$ and finite $3$-palettes $P_n$ with $\lim_{n\to\infty}\lambda(P_n)=t$. For each fixed $r\ge3$, \cref{lem:orientation-lift} gives $\lambda(\operatorname{lift}_r(P_n))=27r^{-r}\lambda(P_n)$, and hence $\lim_{n\to\infty}\lambda(\operatorname{lift}_r(P_n))=27r^{-r}t$. Each term belongs to $\Pu r$ by \cref{lem:finite-values}, while $27r^{-r}t\in[(54/25)r^{-r},(9/4)r^{-r}]\subseteq(r^{-r},4r^{-r})$. This limit is transcendental because $27r^{-r}$ is a nonzero rational. By \cref{thm:attainment}(ii), it does not belong to $\Pu r$. Multiplication by $27r^{-r}$ is injective, so the continuum many choices of $t$ prove (iii).
\end{proof}

The earlier non-jump $12/25$ of Yan and Peng~\cite[Theorem 1.2]{yanpeng} also suffices to prove nonclosedness; the value $4/9$ gives the sharper interval in \cref{thm:nonclosed}(iii).

The proof gives continuum many missing limit points of $\Du_3\cap[0,4/27]$ in $[2/25,1/12]$ and of $\Pu r$ in $[(54/25)r^{-r},(9/4)r^{-r}]$. It does not show that these intervals lie entirely in $\overline{\Du_3\cap[0,4/27]}$ and $\overline{\Pu r}$, respectively, or that a prescribed transcendental number belongs to $\overline{\Du_3}$.

\begin{fact}\label{fact:low-density-examples}
The set $\Du_3\cap[0,4/27]$ is countably infinite, is not discrete, and contains $\sqrt3/18$.
\end{fact}
\begin{proof}
There are only countably many finite palettes up to relabeling. For every $m\ge2$, the palette on $z,b_1,\ldots,b_m$ with tuples $(z,b_i,b_j)$ for $i<j$ has polynomial $x_z\sum_{1\le i<j\le m}x_{b_i}x_{b_j}$. For fixed $x_z$, we have $\sum_{1\le i<j\le m}x_{b_i}x_{b_j}=\frac12((1-x_z)^2-\sum_{i\in[m]}x_{b_i}^2)$. Cauchy--Schwarz therefore shows that the maximum is attained when $x_{b_i}=(1-x_z)/m$ for every $i\in[m]$. Thus its Lagrangian is
\[
\max_{0\le u\le1}\frac12\left(1-\frac1m\right)u(1-u)^2
=\frac2{27}\left(1-\frac1m\right).
\]
These values converge to $2/27$, which is realized by a palette consisting of 2 permutations of one rainbow triple. For an irrational example, select 1 permutation of each of $123,124$ and 2 permutations of each of $134,234$. The polynomial is $x_1x_2(x_3+x_4)+2(x_1+x_2)x_3x_4$. For fixed $t\coloneqq x_3+x_4$, the polynomial is maximized when $x_1=x_2=(1-t)/2$ and $x_3=x_4=t/2$, and its value is $(t-t^3)/4$. Its maximum on $[0,1]$ is $\sqrt3/18<4/27$, attained at $t=1/\sqrt3$.
\end{proof}

\section{A 12-uniform palette construction}\label{sec:compiler}
We construct a 12-uniform palette whose Lagrangian is determined by the Lagrangian of a 3-uniform palette $P$. This construction will give the case $r=12$ of \cref{thm:seed}. Under probability weightings, the base part has maximum $4\cdot12^{-12}$. At an optimal probability weighting of the constructed palette, the additional part contributes $12^{-12}\lambda(P)n^{-8}$ to leading order. Averaging the auxiliary weights will make the dependence on $\lambda(P)$ exact. \Cref{fig:twelve-palette} gives an overview of the construction; the notation is introduced below.

\begin{figure}[!htb]
\centering
\begin{tikzpicture}[x=1cm,y=1cm,font=\small,
  decomposition box/.style={draw=black!40,rounded corners=2pt,align=center,inner sep=4pt},
  decomposition branch/.style={decomposition box,text width=6.6cm,minimum height=1.05cm},
  decomposition factor/.style={decomposition box,text width=1.9cm,minimum height=1.15cm,font=\footnotesize},
  decomposition left factor/.style={decomposition factor,minimum height=1.55cm},
  decomposition link/.style={draw=black!55,line width=.55pt}]
\node[decomposition box,fill=black!3] (root) at (0,0) {$\mathcal H_{n,P}$};
\node[decomposition branch,fill=blue!5] (base) at (-3.7,-1.12)
  {$\Lambda_{\rm base}=2T(\Phi_1+\Phi_2)\prod_{\ell=1}^8t_\ell$\\[2pt]
   \footnotesize $8+3+1=12$ entries};
\node[decomposition branch,fill=orange!7] (add) at (3.7,-1.12)
  {$\Lambda_{\rm add}=\sum_{j=0}^{q-1}R_jD_j(\bz)$\\[2pt]
   \footnotesize $9+3=12$ entries};
\draw[decomposition link] (root.south) -- (0,-.46) -| (base.north);
\draw[decomposition link] (0,-.46) -| (add.north);

\node[decomposition left factor] (t) at (-6.0,-2.91)
  {$\prod_{\ell=1}^8t_\ell$\\8 $t$-colors};
\node[decomposition left factor] (input) at (-1.4,-2.91)
  {$T$\\1 color from $C(P)$};
\node[decomposition left factor] (triple) at (-3.7,-2.91)
  {$\Phi_1+\Phi_2$\\auxiliary triple\\for one monomial};
\draw[decomposition link] (base.south) -- (-3.7,-1.97) -| (t.north);
\draw[decomposition link] (-3.7,-1.97) -| (input.north);
\draw[decomposition link] (-3.7,-1.97) -- (triple.north);

\node[decomposition factor,text width=2.85cm] (connector) at (1.95,-2.91)
  {$R_j$\\$E_{j,t}\in\mathcal M_j$\\9 auxiliary colors};
\node[decomposition factor,text width=2.85cm] (image) at (5.45,-2.91)
  {$D_j(\bz)$\\triple from $P$\\mapped to $\tau_j$ by $\theta$};
\draw[decomposition link] (add.south) -- (3.7,-1.97) -| (connector.north);
\draw[decomposition link] (3.7,-1.97) -| (image.north);
\end{tikzpicture}
\par\vspace{.90cm}
\begin{tikzpicture}[x=1cm,y=1.1cm,line cap=round,line join=round,font=\small,
  every node/.style={inner sep=1pt},
  palettegroup/.style={draw=black!30,fill=black!3,rounded corners=4pt}]
% The circle records the auxiliary cyclic order; only a selected triple is shaded.
\def\paletteauxedge#1#2{%
  \foreach \v in {0,...,7}
    \coordinate (v\v) at ({90-45*\v}:\auxradius);
  \coordinate (vc) at ({-.12*\auxradius},{-.08*\auxradius});
  \coordinate (vd) at ({.43*\auxradius},{-.08*\auxradius});
  \draw[black!30,line width=.4pt] (0,0) circle[radius=\auxradius];
  \foreach \angle in {-67.5,112.5}
    \node[fill=white,inner sep=.2pt,text=black!50,rotate=22.5]
      at (\angle:\auxradius) {$\cdots$};
  \foreach \v in {0,...,7}
    \fill[black!30] (v\v) circle[radius=.7pt];
  \ifnum#1=1
    \def\edgecolor{blue!65!black}\def\edgefill{blue!10}
    \def\first{v0}\def\second{v1}\def\third{v5}
    \def\firstlabel{a_i}\def\secondlabel{b_i}\def\thirdlabel{b_k}
    \def\firstangle{90}\def\secondangle{45}\def\thirdangle{-135}
    \def\firstanchor{south}\def\secondanchor{south west}\def\thirdanchor{north east}
  \else
    \def\edgecolor{blue!65!black}\def\edgefill{blue!10}
    \def\first{v1}\def\second{v2}\def\third{v4}
    \def\firstlabel{b_i}\def\secondlabel{a_{i+1}}\def\thirdlabel{a_k}
    \def\firstangle{45}\def\secondangle{0}\def\thirdangle{-90}
    \def\firstanchor{south west}\def\secondanchor{west}\def\thirdanchor{north}
  \fi
  \ifnum#2>2 \def\second{vc}\fi
  \ifnum#2=2 \def\third{vd}\fi
  \ifnum#2=4 \def\third{vd}\fi
  \draw[draw=\edgecolor,fill=\edgefill,line width=.8pt]
    (\first)--(\second)--(\third)--cycle;
  \foreach \v in {\first,\second,\third}
    \fill[\edgecolor] (\v) circle[radius=1.5pt];
  \node[anchor=\firstanchor] at (\firstangle:\labelradius) {$\firstlabel$};
  \ifnum#2>2
    \def\canchor{north east}\def\cshift{0}
    \ifnum#1=1\ifnum#2=3
      \def\canchor{north west}\def\cshift{.08}
    \fi\fi
    \node[anchor=\canchor] at ({-.12*\auxradius+\cshift},{-.25*\auxradius}) {$c$};
  \else
    \node[anchor=\secondanchor] at (\secondangle:\labelradius) {$\secondlabel$};
  \fi
  \ifnum#2=1 \node[anchor=\thirdanchor] at (\thirdangle:\labelradius) {$\thirdlabel$};\fi
  \ifnum#2=3 \node[anchor=\thirdanchor] at (\thirdangle:\labelradius) {$\thirdlabel$};\fi
  \ifnum#2=2 \node[anchor=north west] at ({.43*\auxradius},{-.25*\auxradius}) {$d$};\fi
  \ifnum#2=4 \node[anchor=north west] at ({.43*\auxradius},{-.25*\auxradius}) {$d$};\fi
}

% Comparable tuple panels follow the left-to-right order of the decomposition tree.
\foreach \shift in {0,7.88} {
  \begin{scope}[shift={(\shift,0)}]
  \foreach \left/\right/\center/\heading in {
    .10/1.68/.89/{$t$-colors},2.28/5.39/3.835/{Auxiliary colors},6.03/7.32/6.675/{$C(P)$}} {
    \draw[palettegroup] (\left,5.12) rectangle (\right,7.26);
    \node[anchor=south] at (\center,7.43) {\heading};
  }
  \end{scope}
}
% Eight distinct t-colors, one auxiliary triple, and one input color.
\foreach \row in {0,...,3} {
  \foreach \col in {0,1} {
    \pgfmathtruncatemacro{\idx}{1+\row+4*\col}
    \fill[blue!65!black] ({.27+.77*\col},{6.74-.43*\row}) circle[radius=1.65pt];
    \node[anchor=west,font=\scriptsize] at ({.38+.77*\col},{6.74-.43*\row}) {$t_{\idx}$};
  }
}
\begin{scope}[xshift=3.95cm]
\foreach \xx/\yy in {2.31/6.88,3.12/6.76,2.36/5.44,3.08/5.55}
  \fill[black!30] (\xx,\yy) circle[radius=.8pt];
\fill[blue!65!black] (2.725,6.15) circle[radius=1.65pt];
\node[anchor=north] at (2.725,6.00) {$u$};
\end{scope}
\begin{scope}[shift={(3.835,6.12)},font=\scriptsize]
  \def\auxradius{.76}\def\labelradius{.88}
  \paletteauxedge{1}{2}
\end{scope}

\begin{scope}[shift={(-7.88,-3.23)}]
% The four monomial types in each auxiliary cubic polynomial.
\foreach \xx/\kind in {8.61/1,10.52/2,12.43/3,14.34/4} {
  \foreach \yy/\row in {6.62/1,4.90/2} {
    \begin{scope}[shift={(\xx,\yy)},font=\scriptsize]
      \def\auxradius{.49}\def\labelradius{.60}
      \paletteauxedge{\row}{\kind}
    \end{scope}
  }
}

\end{scope}

\begin{scope}[shift={(7.88,4.28)}]
\node[text=black!45] at (.89,1.85) {---};
% The three marked positions are ordered entries and may repeat a color.
\begin{scope}[xshift=3.95cm]
\draw[draw=orange!75!black,fill=orange!5,rounded corners=4pt,line width=.7pt]
  (2.27,1.30) rectangle (3.16,2.72);
\foreach \yy/\idx in {2.47/1,2.01/2,1.55/3} {
  \fill[orange!80!black] (2.50,\yy) circle[radius=1.65pt];
  \node[anchor=west,font=\scriptsize] at (2.65,\yy) {$u_{\idx}$};
}
\end{scope}
\begin{scope}[xshift=-1.73cm]
\draw[draw=orange!75!black,fill=orange!5,rounded corners=5pt,line width=.7pt]
  (4.26,1.19) rectangle (6.87,2.60);
\foreach \xx in {4.69,5.565,6.44}
  \foreach \yy in {1.47,1.89,2.31}
    \fill[orange!80!black] (\xx,\yy) circle[radius=1.65pt];
\node[font=\scriptsize,anchor=north] at (5.565,.68) {$E_{j,t}\in\mathcal M_j$};

\end{scope}

\end{scope}

% Connector structure below the additional tuple, without shrinking the labels.
\begin{scope}[shift={(7.88,1.05)},font=\scriptsize,
  connector A/.style={circle,draw=blue!45!black,fill=blue!45!black,inner sep=0pt,minimum size=3.7pt},
  connector B/.style={rectangle,draw=black!60,fill=black!60,inner sep=0pt,minimum size=3.7pt},
  connector omitted/.style={circle,draw=black!35,fill=white,inner sep=0pt,minimum size=3.7pt},
  connector accent/.style={draw=orange!65!black,line width=.55pt},
  connector brace/.style={decorate,decoration={brace,amplitude=3pt,mirror},draw=black!50,line width=.35pt}]
\node[anchor=east] at (.34,3.02) {$E_{0,t}$};
\node[anchor=east] at (.34,1.43) {$E_{j,t}$};
\node[anchor=east,text=black!60] at (.34,1.08) {$j>0$};
\foreach \yy in {3.02,1.43} {
  \draw[black!20,line width=.35pt] (.61,\yy)--(5.72,\yy);
  \foreach \i in {0,...,8} {
    \pgfmathsetmacro{\xx}{.80+.59*\i}
    \ifdim\yy pt<2pt
      \ifnum\i=4
        \node[connector omitted] at (\xx,\yy) {};
      \else
        \ifodd\i\node[connector B] at (\xx,\yy) {};
        \else\node[connector A] at (\xx,\yy) {};\fi
      \fi
    \else
      \ifodd\i\node[connector B] at (\xx,\yy) {};
      \else\node[connector A] at (\xx,\yy) {};\fi
    \fi
  }
}
\node[text=black!45] at (6.04,1.43) {$\cdots$};
\draw[black!20,line width=.35pt] (6.40,1.43)--(7.12,1.43);
\node[connector A] at (6.85,1.43) {};
\node[anchor=north] at (6.85,1.26) {$A_{t+2+9j}$};
\draw[connector accent] (3.16,3.02) circle[radius=4pt];
\draw[connector accent] (6.85,1.43) circle[radius=4pt];
\draw[connector accent,-{Stealth[length=1.3mm]},rounded corners=5pt]
  (3.16,2.85)--(3.16,2.22)--(6.85,2.22)--(6.85,1.62);
\node[anchor=south,text=orange!65!black] at (5.0,2.28) {$18j$ positions later};
\node[text=black!45] at (3.16,1.06) {omitted};
\draw[connector brace] (.63,.78)--(2.74,.78);
\draw[connector brace] (3.61,.78)--(7.08,.78);
\node[anchor=north,text=black!65] at (1.685,.59) {$2$ $A$-colors, $2$ $B$-colors};
\node[anchor=north,text=black!65] at (5.345,.59) {$3$ $A$-colors, $2$ $B$-colors};
\end{scope}
\end{tikzpicture}
\caption{Construction of $\mathcal H_{n,P}$. The lower-left triples represent monomials of $\Phi_1$ (top) and $\Phi_2$ (bottom). At lower right, $A_{t+2}$ is replaced by $A_{t+2+9j}$; circles and squares denote $A$- and $B$-colors. The hollow marker denotes a color omitted from the connector, and thin lines indicate cyclic order. The triple from $P$ may repeat a color.}
\label{fig:twelve-palette}
\end{figure}
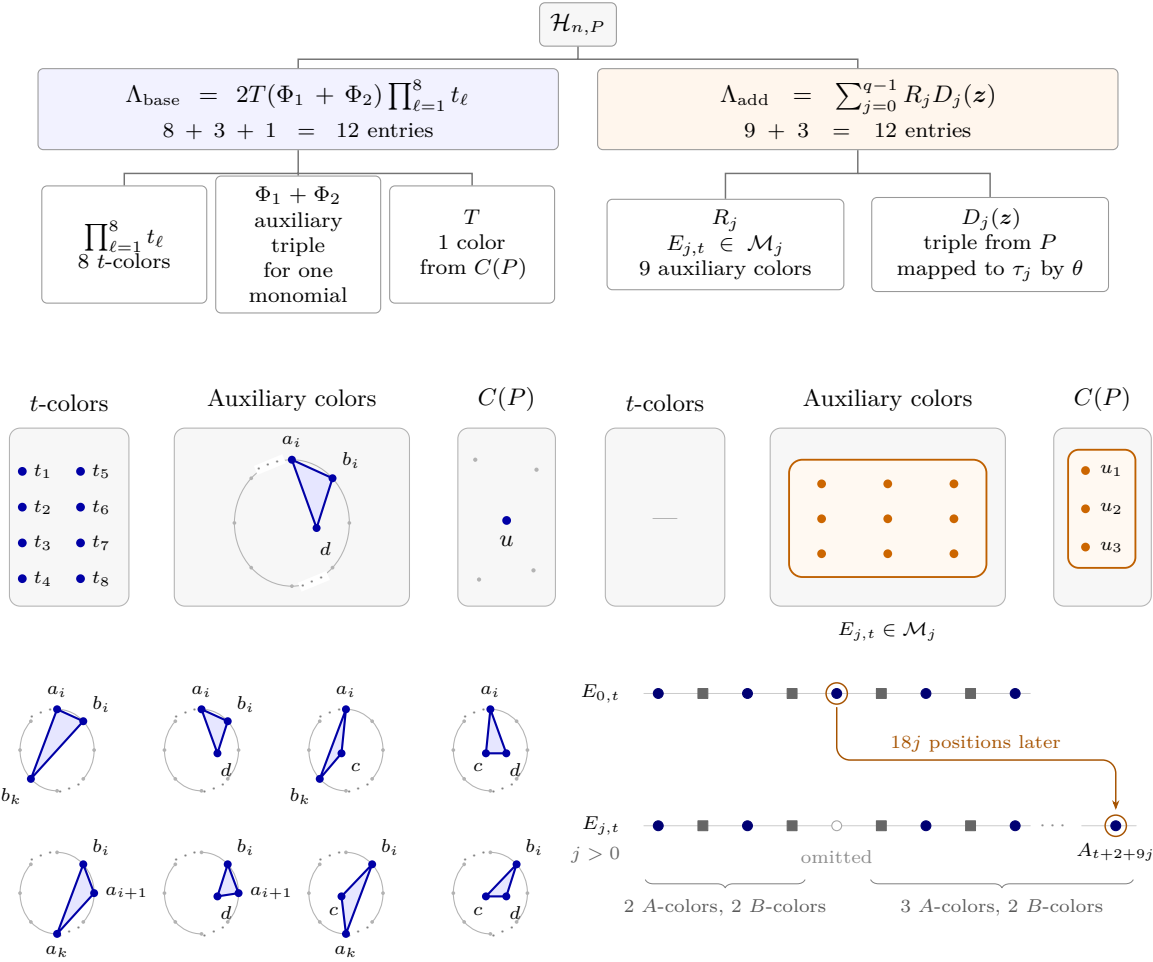

A sequence of finite palettes with homomorphisms $P_{i+1}\to P_i$ is a \emph{projective sequence}. Its Lagrangians are nonincreasing. Li, Liu, and Pikhurko proved the following realization theorem.
\begin{theorem}[{\cite[Proposition 3.1 and Section 6]{llp}}]\label{thm:projective}
If finite $r$-palettes satisfy $P_{i+1}\to P_i$ and $\lim_{i\to\infty}\lambda(P_i)=x$, then $x\in\Pu r$.
\end{theorem}
The cited result is stated for strictly decreasing Lagrangians. If the Lagrangians are eventually constant, the conclusion follows from \cref{lem:finite-values}; otherwise, extract a strictly decreasing subsequence and compose the intermediate homomorphisms.

The same paper supplies projective sequences whose limiting Lagrangians fill an interval and whose maps to a fixed palette are compatible with the transition maps.
\begin{proposition}[{\cite[Theorem 5.3 and Section 6]{llp}}]\label{prop:input}
There are a finite $3$-palette $P_*$ without monochromatic tuples and constants $0<a<b<1$ such that every $x\in[a,b]$ is realized by a projective sequence $P_*\longleftarrow P_1\longleftarrow P_2\longleftarrow\cdots$ with $\lim_{i\to\infty}\lambda(P_i)=x$. There is also a $\rho\in(0,1)$ such that every point of $[\rho,1]$ has a projective realization by finite $3$-palettes.
\end{proposition}
The face order in~\cite{llp} is the reverse of the order fixed in \cref{sec:palettes}. Reversing all its palettes gives the stated convention. For each projective sequence used below that starts at $P_*$, take its maps to $P_*$ to be the compositions along that sequence; they therefore commute with the transition maps.

Fix the palette $P_*$ and the interval $[a,b]$ from \cref{prop:input}. Enumerate the admissible triples of $P_*$ as $\tau_0,\ldots,\tau_{q-1}$. For a finite $3$-palette $P$ with a homomorphism $\theta\colon P\to P_*$, define
\[
D_j(\bz)\coloneqq
\sum_{\substack{(u,v,w)\in P\\(\theta(u),\theta(v),\theta(w))=\tau_j}}
z_uz_vz_w
\]
for $0\le j<q$, where $\bz\coloneqq(z_c)_{c\in C(P)}$. Each admissible triple is counted separately, even when different triples give the same monomial. Thus $\Lambda_P(\bz)=\sum_{j=0}^{q-1}D_j(\bz)$. For a nonnegative weighting $\bz$ of $P$ with total weight $T$, we have
\begin{equation}\label{eq:piece-bounds}
0\le D_j(\bz)\le\frac4{27}T^3,
\quad\text{and}\quad
\sum_{j=0}^{q-1}D_j(\bz)\le T^3.
\end{equation}
For the first bound, write $\tau_j=(d_1,d_2,d_3)$ and sum the weights of colors with the same image under $\theta$. This gives
\[
D_j(\bz)\le\prod_{\ell=1}^3\left(\sum_{\substack{c\in C(P)\\\theta(c)=d_\ell}}z_c\right).
\]
The color multiplicities of $\tau_j$ are either $1+1+1$ or $2+1$, since $P_*$ has no monochromatic triple. For these two patterns, the product is at most $T^3/27$ and $4T^3/27$, respectively, by AM--GM, proving the first bound. The second bound follows because the admissible triples are a subset of all ordered triples.

Fix an integer $n\in18\N$ with $n\ge36(q+1)$.

Take auxiliary colors $A_i,B_i$ for $i\in\mathbb Z/n\mathbb Z$, and 2 further colors with weights $c,d$. Read their subscripts modulo $n$. Write $a_i,b_i$ for the corresponding weights, and put $A\coloneqq\sum_{i=0}^{n-1} a_i$ and $B\coloneqq\sum_{i=0}^{n-1} b_i$. We use the following edge polynomials of simple $3$-graphs from~\cite[Section 2.3]{ordinary}.
\[
\begin{aligned}
\Phi_1&\coloneqq\sum_{i=0}^{n-1} a_i(b_i+c)(B-b_i+d)=\sum_{\substack{i,j\in\mathbb Z/n\mathbb Z\\j\ne i}}a_i b_i b_j
+\sum_{\substack{i,j\in\mathbb Z/n\mathbb Z\\j\ne i}}a_i b_j c
+\sum_{i=0}^{n-1}a_i b_i d
+\sum_{i=0}^{n-1}a_i c d,\\
\Phi_2&\coloneqq\sum_{i=0}^{n-1} b_i(a_{i+1}+c)(A-a_{i+1}+d)=\sum_{\substack{i,j\in\mathbb Z/n\mathbb Z\\j\ne i+1}}b_i a_{i+1}a_j
+\sum_{\substack{i,j\in\mathbb Z/n\mathbb Z\\j\ne i+1}}b_i a_j c
+\sum_{i=0}^{n-1}b_i a_{i+1}d
+\sum_{i=0}^{n-1}b_i c d.
\end{aligned}
\]
The 4 sums in each expansion give the 4 types of auxiliary edges. In the factored expressions, the distinguished variable and the 2 parenthesized factors involve pairwise disjoint sets of colors. The distinguished $a_i$ or $b_i$ determines the summand, so each $\Phi_i$ has coefficients 0 or 1.

To define the connecting families, set $v_{2i}\coloneqq A_i$ and $v_{2i+1}\coloneqq B_i$, with subscripts read modulo $2n$. For $0\le j<q$, let $\mathcal M_j$ consist of the $n$ sets
\begin{equation}\label{eq:connector-edges}
E_{j,t}\coloneqq
\{v_{2t+\ell}\colon\ell\in\{0,1,2,3,5,6,7,8\}\}
\cup\{v_{2t+4+18j}\},
\qquad t\in\mathbb Z/n\mathbb Z.
\end{equation}
The edge polynomial of $\mathcal M_j$ is
\[
R_j\coloneqq\sum_{t=0}^{n-1}
a_ta_{t+1}a_{t+2+9j}a_{t+3}a_{t+4}
b_tb_{t+1}b_{t+2}b_{t+3},
\]
where all subscripts are read modulo $n$. Each monomial corresponds to the edge $E_{j,t}$. Each edge is obtained from 9 consecutive auxiliary colors by moving their middle $A$-color through $18j$ positions; see the lower-right diagram in \cref{fig:twelve-palette}.

Take 8 further colors $t_1,\ldots,t_8$, disjoint from all auxiliary colors and the colors of $P$, using the same symbols for their weights. We construct a $12$-palette $\mathcal H_{n,P}$ with polynomial $\Lambda_{\mathcal H_{n,P}}=\Lambda_{\rm base}+\Lambda_{\rm add}$, where
\[
\Lambda_{\rm base}\coloneqq2T(\Phi_1+\Phi_2)\prod_{\ell=1}^8t_\ell,
\quad\text{and}\quad
\Lambda_{\rm add}\coloneqq\sum_{j=0}^{q-1}R_jD_j(\bz).
\]
Here $T\coloneqq\sum_{c\in C(P)}z_c$ is the total weight of the colors of $P$. The construction depends on the chosen homomorphism $\theta$, which we suppress from the notation.

The admissible tuples are specified as follows.
\begin{enumerate}[beginpenalty=10000]
\item For the base part, put 8 $t$-colors first, the 3 auxiliary colors next, and 1 color of $P$ last. Choose 4 distinct permutations of the $t$-colors, using 2 for each $\Phi_i$; fix an arbitrary ordering of every auxiliary edge. For each fixed last color from $C(P)$, an auxiliary edge whose monomial occurs in both $\Phi_1$ and $\Phi_2$ gives 4 different tuples.
\item For the additional part, take each connector $E_{j,t}\in\mathcal M_j$, fix an ordering of its colors, and append each admissible triple of $P$ whose image under $\theta$ is $\tau_j$. We verify below that the connecting families are simple and pairwise edge-disjoint, so each connecting edge identifies its index $j$. These tuples contain no $t$-colors and therefore cannot coincide with base tuples.
\end{enumerate}
For each $n$, choose the 4 permutations of the $t$-colors and the orders of the auxiliary edges and connectors once, and use them for every input palette and every homomorphism to $P_*$. Thus the prescribed tuples are distinct, including when an admissible triple of $P$ repeats a color.

Let $\phi\colon P'\to P$ be a homomorphism, and let $\theta'\colon P'\to P_*$ and $\theta\colon P\to P_*$ satisfy $\theta'=\theta\circ\phi$. Extending $\phi$ by fixing all auxiliary and $t$-colors gives a homomorphism $\mathcal H_{n,P'}\to\mathcal H_{n,P}$. The identity $\theta'=\theta\circ\phi$ preserves the partition of the admissible triples according to their images in $P_*$ and hence all admissible tuples of the constructed palettes.

We optimize over probability weightings of $\mathcal H_{n,P}$. The next proposition gives the exact dependence of $\lambda(\mathcal H_{n,P})$ on $\lambda(P)$.

\begin{samepage}
\begin{proposition}\label{prop:exact-reader}
For every sufficiently large $n\in18\N$ with $n\ge36(q+1)$, there is a continuous function $\Psi_n\colon[0,\infty)\to[0,\infty)$, strictly increasing on $(0,\infty)$, such that the following hold for every finite $3$-palette $P$ with $P\ne\varnothing$ equipped with a homomorphism $\theta\colon P\to P_*$.
\begin{enumerate}
\item Every optimal weighting of $\mathcal H_{n,P}$ is constant within each auxiliary class, that is, $a_i=A/n$ and $b_i=B/n$ for all $i$.
\item We have $\lambda(\mathcal H_{n,P})=\Psi_n(\lambda(P))$; in particular, this value is independent of the choice of $\theta$.
\item We have $\lambda(\mathcal H_{n,P})=12^{-12}(4+\lambda(P)n^{-8})+O(n^{-12})$, where the error is nonnegative and the implicit constant is uniform over these choices of $P$ and $\theta$.
\end{enumerate}
\end{proposition}
\end{samepage}

The proof first estimates $\Phi_1+\Phi_2$ and the contributions from the connecting families. Within the resulting range of weights, balancing the auxiliary classes increases the base part by at least a constant multiple of $n^{-1}$ times the sum of squared deviations from the class averages, whereas the possible loss in the additional part is at most $O(n^{-7})$ times the same sum. Thus every optimal weighting is constant within each auxiliary class, and the maximum depends on $P$ only through $\lambda(P)$.

\begin{proof}[Proof of \cref{prop:exact-reader}]

Write $S\coloneqq A+B+c+d$ for the total auxiliary weight. Put $\rho_A\coloneqq(A+d-c)/2$ and $\rho_B\coloneqq(B+d-c)/2$, and define
\[
E\coloneqq\sum_{i=0}^{n-1} a_i(b_i-\rho_B)^2+\sum_{i=0}^{n-1} b_i(a_{i+1}-\rho_A)^2,
\quad\text{and}\quad
\Delta\coloneqq\frac{2S^3}{27}-\Phi_1-\Phi_2.
\]
Here $\Delta$ is the \emph{deficit} of $\Phi_1+\Phi_2$ from its upper bound $2S^3/27$.
Completing the square in each summand of $\Phi_1$ and $\Phi_2$ gives
\[
\Phi_1=\frac{A(S-A)^2}{4}-\sum_{i=0}^{n-1} a_i(b_i-\rho_B)^2,
\quad\text{and}\quad
\Phi_2=\frac{B(S-B)^2}{4}-\sum_{i=0}^{n-1} b_i(a_{i+1}-\rho_A)^2.
\]
In particular, if $S>0$, then we have
\begin{equation}\label{eq:cyclic-variance}
\max\{\Phi_1,\Phi_2\}\le\left(\frac{S}{3}\right)^3,
\qquad E\le\Delta,
\quad\text{and}\quad
\left(A-\frac{S}{3}\right)^2+\left(B-\frac{S}{3}\right)^2\le\frac{12\Delta}{S}.
\end{equation}
Indeed, we have
\[
\frac{S^3}{27}-\frac{A(S-A)^2}{4}
=\frac14\left(A-\frac{S}{3}\right)^2\left(\frac{4S}{3}-A\right).
\]
Since $A\le S$, we have $4S/3-A\ge S/3$. Applying the same identity with $B$ in place of $A$ gives the stated bounds. The bounds $\Phi_i\le(S/3)^3$ for $i\in\{1,2\}$ are attained simultaneously at
\begin{equation}\label{eq:cyclic-test}
a_i=b_i=\frac{S}{3n},
\qquad c=\frac S3\left(1-\frac{1}{n}\right),\qquad d=\frac S{3n}.
\end{equation}

For each $i\in\mathbb Z/n\mathbb Z$, put $X_i\coloneqq\{A_i,B_i,A_{i+1},B_{i+1}\}$. These 4 colors are distinct since $n\ge2$. Put $h\coloneqq\max\{0,\rho_A,\rho_B\}$. The next claim bounds the individual weights when $\Delta<h^3/28$ and the sum $\sum_{i=0}^{n-1} a_i b_i a_{i+1}b_{i+1}$ otherwise.

\begin{claim}\label{lem:cyclic-control}
If $S>0$ and $\Delta<h^3/28$, then $a_i,b_i\in(S/(27n),9S/(4n)]$ for all $i$. If $\Delta\ge h^3/28$, then $\sum_{i=0}^{n-1} a_i b_i a_{i+1}b_{i+1}\le\frac{1675}{16}A\Delta$.
\end{claim}
\begin{proof}[Proof of the claim]
For the first assertion, we have $0<h\le S/2$. By~\eqref{eq:cyclic-variance}, $|\rho_A-\rho_B|^2\le6\Delta/S<3h^2/28<h^2/9$, so $2h/3<\rho_A,\rho_B\le h$. We also have $A+B\ge2S/3-\sqrt{24\Delta/S}>S(2/3-\sqrt{3/28})>S/3$. Let $\widehat u_1,\ldots,\widehat u_{2n}$ be the cyclic sequence alternating $a_i/\rho_A$ and $b_i/\rho_B$. Read the subscripts modulo $2n$. This sequence has sum greater than $2/3$, and
\[
K\coloneqq\sum_{i\in[2n]}\widehat u_i(1-\widehat u_{i+1})^2\le\frac{27E}{8h^3}<\frac{27}{224}<\frac18.
\]
Here the two alternating types of summands in $K$ are $a_i(b_i-\rho_B)^2/(\rho_A\rho_B^2)$ and $b_i(a_{i+1}-\rho_A)^2/(\rho_B\rho_A^2)$. Their denominators are at least $(2h/3)^3$, giving the bound in terms of $E$. For any cyclic nonnegative sequence with sum at least $1/2$ and $K<1/8$, every coordinate is at least $1/2$. Otherwise, either all coordinates are below $1/2$, giving $K>\frac14\sum_{i\in[2n]}\widehat u_i\ge1/8$, or a cyclic transition from a coordinate at least $1/2$ to one below $1/2$ contributes more than $1/8$. Using $\widehat u_{i-1}\ge1/2$ in $\widehat u_{i-1}(1-\widehat u_i)^2\le K$ gives $|\widehat u_i-1|\le\sqrt{2K}<1/2$. Thus $h/3<a_i,b_i<3h/2$. Summing and using $S/3<A+B\le S$ yields $S/(9n)<h<3S/(2n)$, which gives the first assertion.

For the second assertion, write $(u_0,u_1,u_2,u_3)\coloneqq(a_i,b_i,a_{i+1},b_{i+1})$ for the consecutive weights of $X_i$. If the last 3 weights are at most $5h/4$, their product with $u_0$ is at most $(125/64)a_i h^3$. Otherwise choose a largest one, say $u_k>5h/4$ for $k\in\{1,2,3\}$. For $k\in\{1,3\}$, the sum $E$ contains $u_{k-1}(u_k-\rho_B)^2$; for $k=2$, it contains $u_1(u_2-\rho_A)^2$. In either case this term is at least $u_{k-1}u_k^2/25$, since $\rho_A,\rho_B\le h<4u_k/5$. If $u_0u_1u_2u_3=0$, the required bound is immediate. Otherwise, the ratio of $u_0u_1u_2u_3$ to $u_{k-1}u_k^2$ is at most $A$. Indeed, for $k=1,2,3$ it is bounded respectively by $u_2,u_0,u_0$. The term $a_i(b_i-\rho_B)^2$ can be used only for $X_{i-1}$ or $X_i$, whereas $b_i(a_{i+1}-\rho_A)^2$ can be used only for $X_i$. Counting each summand of $E$ at most twice therefore gives
\[
\sum_{i=0}^{n-1} a_i b_i a_{i+1}b_{i+1}\le50AE+\frac{125}{64}Ah^3
\le\left(50+\frac{125\cdot28}{64}\right)A\Delta
=\frac{1675}{16}A\Delta.\qedhere
\]
\end{proof}

We next bound the edge polynomials $R_j$, including their first derivatives at the weighting that is constant within each class.

\begin{claim}\label{lem:connectors}
The following statements hold.
\begin{enumerate}
\item The $\mathcal M_j$ are pairwise edge-disjoint simple $9$-graphs. The $9$-graph $\bigcup_{j=0}^{q-1}\mathcal M_j$ is $9$-partite, and $\sum_{j=0}^{q-1}R_j\le(S/9)^9$.
\item Every $A$-color has degree 5 and every $B$-color degree 4 in each $\mathcal M_j$. At the weighting $a_i=A/n$, $b_i=B/n$, we have $R_j=A^5B^4/n^8$, $\partial_{a_i}R_j=5(A/n)^4(B/n)^4$, and $\partial_{b_i}R_j=4(A/n)^5(B/n)^3$.
\item If $\Delta\ge h^3/28$, then $R_j\le\frac{1675}{1728}A^4B^2\Delta$.
\end{enumerate}
\end{claim}
\begin{proof}[Proof of the claim]
For (i), observe that the 9 indices defining $E_{j,t}$ in~\eqref{eq:connector-edges} are distinct modulo $2n$. For $j=0$ they form 1 block of 9 consecutive positions. For $j>0$ they form 2 blocks of 4 consecutive positions and 1 isolated position. Of the 2 blocks of length 4, exactly 1 starts at an even index. That block starts at $2t$, which determines $t$, and the isolated index $2t+4+18j$ then determines $j$. The inequality $n\ge36(q+1)$ prevents the isolated index from being adjacent to either block in the cyclic order modulo $2n$. This proves simplicity and disjointness, both within and between the families.

Partition the auxiliary colors $v_0,\ldots,v_{2n-1}$ according to their indices modulo 9. Since $18j\equiv0\pmod9$ and $2n$ is divisible by 9, every connecting edge meets each part exactly once. Standard results on the Lagrangians of complete multipartite hypergraphs therefore give $\sum_{j=0}^{q-1}R_j\le(S/9)^9$, completing (i).

For (ii), translation by 2 positions preserves $\mathcal M_j$ and acts transitively on each of the $A$- and $B$-classes, so the degrees are constant within each class. Its $n$ edges contain 5 $A$-colors and 4 $B$-colors each, giving $5n$ and $4n$ incidences, respectively. Dividing by the $n$ colors in each class gives degrees 5 and 4. At the balanced weighting, each of the $n$ edge monomials equals $(A/n)^5(B/n)^4$, yielding the formula for $R_j$. Differentiating with respect to an $A$-color leaves 4 $A$-weights and 4 $B$-weights in each of its 5 incident edges; for a $B$-color, it leaves 5 $A$-weights and 3 $B$-weights in each of its 4 incident edges. This gives the two derivative formulas.

For (iii), $E_{j,t}$ contains $X_t$. The set $E_{j,t}\setminus X_t$ consists of 3 $A$-colors and 2 $B$-colors. Hence
\[
R_j\le\left(\frac{A}{3}\right)^3\left(\frac{B}{2}\right)^2\sum_{t=0}^{n-1} a_t b_t a_{t+1}b_{t+1}
\le\frac{1675}{1728}A^4B^2\Delta.
\]
The second inequality follows from \cref{lem:cyclic-control}.
\end{proof}

AM--GM, \eqref{eq:piece-bounds}, \eqref{eq:cyclic-variance}, and \cref{lem:connectors}(i) give
\begin{equation}\label{eq:base-add-four}
\Lambda_{\rm base}\le4T\left(\frac{S}{3}\right)^3\prod_{\ell=1}^8t_\ell\le4\cdot12^{-12},
\quad\text{and}\quad
\Lambda_{\rm add}\le4\left(\frac{S}{9}\right)^9\left(\frac{T}{3}\right)^3\le4\cdot12^{-12}.
\end{equation}
For a lower bound, take the test weighting with $t_\ell=T=1/12$, $S=1/4$, the auxiliary weights in~\eqref{eq:cyclic-test}, and an optimal probability weighting of $P$ scaled to total weight $T$. It yields
\begin{equation}\label{eq:compiler-test}
\lambda(\mathcal H_{n,P})\ge12^{-12}\bigl(4+\lambda(P)n^{-8}\bigr)>4\cdot12^{-12}.
\end{equation}
Hence both parts of the polynomial are positive at every optimal weighting, as are all $t$-weights and the group totals $S,T$.

Fix an optimal probability weighting $\bx$ and put $\eta\coloneqq\Lambda_{\rm add}/\lambda(\mathcal H_{n,P})\in(0,1)$. We first estimate the total weights of the color classes. At an optimal weighting, all positive coordinates have the same partial derivative. Euler's identity and the total weight 1 show that this common value is $12\lambda(\mathcal H_{n,P})$. Each base monomial contains every $t_\ell$ to exponent 1, whereas additional monomials contain no $t$-variables. Each base monomial has total degree 3 in the auxiliary variables and total degree 1 in the weights of the colors of $P$; for each additional monomial, these total degrees are 9 and 3, respectively. Multiplying the derivative equations by the corresponding weights and summing over each group, with zero coordinates contributing zero, gives $12t_\ell\lambda(\mathcal H_{n,P})=\Lambda_{\rm base}$, $12S\lambda(\mathcal H_{n,P})=3\Lambda_{\rm base}+9\Lambda_{\rm add}$, and $12T\lambda(\mathcal H_{n,P})=\Lambda_{\rm base}+3\Lambda_{\rm add}$. Dividing by $12\lambda(\mathcal H_{n,P})$ and using the definition of $\eta$, we obtain
\begin{equation}\label{eq:compiler-euler}
t_\ell=\frac{1-\eta}{12},\qquad S=\frac{1+2\eta}{4},\qquad T=\frac{1+2\eta}{12}.
\end{equation}
Since $\lambda(\mathcal H_{n,P})>4\cdot12^{-12}$, equation~\eqref{eq:base-add-four} gives $\eta<12^{12}\Lambda_{\rm add}/4\le12^{12}(S/9)^9(T/3)^3$. Applying AM--GM to 9 copies of $S/9$ and 3 copies of $T/3$ now gives $\eta<(S+T)^{12}=((4+8\eta)/12)^{12}$.
The function $\eta\mapsto((4+8\eta)/12)^{12}-\eta$ is strictly convex, vanishes at 1, and is negative at $10^{-5}$. Convexity makes it negative throughout $[10^{-5},1)$, whereas the preceding inequality requires it to be positive at $\eta$. Since $0<\eta<1$, we obtain $\eta<10^{-5}$.

Keeping the deficit $\Delta$ in the base polynomial gives
\begin{equation}\label{eq:base-retained}
\Lambda_{\rm base}=2T\left(\frac{2S^3}{27}-\Delta\right)\prod_{\ell=1}^8t_\ell\le4\cdot12^{-12}-2T\Delta\prod_{\ell=1}^8t_\ell.
\end{equation}
Since $\Lambda_{\rm base}=(1-\eta)\lambda(\mathcal H_{n,P})>4\cdot12^{-12}(1-\eta)$, we have $\Delta<2\cdot12^{-12}\eta/(T\prod_{\ell=1}^8t_\ell)$. Equations~\eqref{eq:compiler-euler} and~\eqref{eq:cyclic-variance} imply
\[
S<\frac{3.00006}{12},\qquad T<\frac{1.00002}{12},\qquad \prod_{\ell=1}^8t_\ell>\frac{0.9999}{12^8},\quad\text{and}\quad\max\{A,B\}<\frac{1.01}{12}.
\]
For the bound on $\max\{A,B\}$, $12^3\Delta<2\cdot10^{-5}/0.9999$ and $\max\{|A-S/3|,|B-S/3|\}\le\sqrt{12\Delta/S}$ suffice.

We now show that $\Delta<h^3/28$. If $\Delta\ge h^3/28$, then \cref{lem:connectors} and $\sum_{j=0}^{q-1}D_j(\bz)\le T^3$ give
\[
\lambda(\mathcal H_{n,P})\le4\cdot12^{-12}-\left(2T\prod_{\ell=1}^8t_\ell-\frac{1675}{1728}A^4B^2T^3\right)\Delta\le4\cdot12^{-12},
\]
since, after multiplication by $12^9$, the coefficient in parentheses is greater than $2\cdot0.9999-(1675/1728)(1.01)^6(1.00002)^3>0.94$. Here we used $T\ge1/12$ from~\eqref{eq:compiler-euler} as well as the preceding bounds. This contradicts~\eqref{eq:compiler-test}. Hence $\Delta<h^3/28$, and \cref{lem:cyclic-control} gives $a_i,b_i\in(1/(108n),7/(12n))$.
Since every connector has 9 vertices and there are $n$ connectors in each family, we have $R_j\le n(7/(12n))^9=(7/12)^9n^{-8}$. Since $\sum_{j=0}^{q-1}D_j(\bz)\le T^3$, we also have $\Lambda_{\rm add}=O(n^{-8})$. By~\eqref{eq:base-retained} and~\eqref{eq:compiler-test}, we have $2T\Delta\prod_{\ell=1}^8t_\ell<\Lambda_{\rm add}$. Since $T\prod_{\ell=1}^8t_\ell$ is bounded below by an absolute positive constant, this gives $\Delta=O(n^{-8})$. Also $\eta<12^{12}\Lambda_{\rm add}/4=O(n^{-8})$.
All implicit constants in this proof are absolute; in particular, they are independent of $P$ and $q$.

We next show that the weights are constant within each auxiliary class. Put $\xi_i\coloneqq a_i-A/n$ and $\zeta_i\coloneqq b_i-B/n$, so $\sum_{i=0}^{n-1}\xi_i=\sum_{i=0}^{n-1}\zeta_i=0$. Write $V\coloneqq\sum_{i=0}^{n-1}\xi_i^2+\sum_{i=0}^{n-1}\zeta_i^2$. Keep $A,B,c,d$, the weighting of $P$, and the $t$-weights fixed, and equalize the weights within each auxiliary class. We denote this balanced weighting by $\bar{\bx}$.

Since $b_i>1/(108n)$, we have $E\ge(108n)^{-1}\sum_{i=0}^{n-1}(a_{i+1}-\rho_A)^2\ge\frac1{108}(A/n-\rho_A)^2$, where the second inequality is Cauchy--Schwarz. Together with $E\le\Delta$, this gives $|A/n-\rho_A|\le\sqrt{108\Delta}$; the same argument gives $|B/n-\rho_B|\le\sqrt{108\Delta}$. Expanding the square-completion formulas gives
\begin{align*}
\Phi_1(\bar{\bx})-\Phi_1(\bx)
&=\sum_{i=0}^{n-1} a_i\zeta_i^2+2\left(\frac{B}{n}-\rho_B\right)\sum_{i=0}^{n-1}\xi_i\zeta_i,\\
\text{and}\qquad \Phi_2(\bar{\bx})-\Phi_2(\bx)
&=\sum_{i=0}^{n-1} b_i\xi_{i+1}^2+2\left(\frac{A}{n}-\rho_A\right)\sum_{i=0}^{n-1}\zeta_i\xi_{i+1}.
\end{align*}
Cauchy--Schwarz gives $|\sum_{i=0}^{n-1}\xi_i\zeta_i|+|\sum_{i=0}^{n-1}\zeta_i\xi_{i+1}|\le V$. The two cross terms therefore have sum at least $-2\sqrt{108\Delta}V$. Since $a_i,b_i>1/(108n)$, the sum of the two differences is at least $(1/(108n)-2\sqrt{108\Delta})V$. Using $\Delta=O(n^{-8})$, we obtain a lower bound of $V/(216n)$ for all sufficiently large $n$. Multiplying by $2T\prod_{\ell=1}^8t_\ell$ shows that the base part increases by at least $\frac{TV}{108n}\prod_{\ell=1}^8t_\ell$.

To bound the possible loss in the additional part, we use \cref{lem:connectors}(ii), which states that the partial derivatives of each $R_j$ at the balanced weighting are constant within each auxiliary class. The weight changes sum to zero in each class, so $\sum_{i=0}^{n-1}\bigl(\partial_{a_i}R_j(\bar{\bx})\xi_i+\partial_{b_i}R_j(\bar{\bx})\zeta_i\bigr)=0$. Thus the first-order term in the balancing direction vanishes at $\bar{\bx}$. Along the segment from the original point to the balanced weighting, each weight $a_i$ or $b_i$ is at most $7/(12n)$. Each color occurs in at most 5 edges of $\mathcal M_j$, with 8 other vertices per edge. For any 2 distinct colors in an edge, that edge contributes the product of the other 7 weights to the corresponding mixed second derivative of $R_j$. Thus every row sum of the Hessian of $R_j$ with respect to $(a_i,b_i)_{i=0}^{n-1}$ is at most $40(7/(12n))^7$. The Hessian is symmetric with nonnegative entries, so this also bounds its operator norm. Expanding at $\bar{\bx}$, where the first-order term vanishes, and using $\|\bx-\bar{\bx}\|_2^2=V$, Taylor's formula gives $|R_j(\bx)-R_j(\bar{\bx})|\le\frac12\cdot40(7/(12n))^7V=O(n^{-7})V$. Using $\sum_{j=0}^{q-1}D_j\le T^3$ once more, the additional part can decrease by at most $O(n^{-7})V$. Since $T\prod_{\ell=1}^8t_\ell$ is bounded below by an absolute positive constant, this loss is smaller than the gain in the base part for all sufficiently large $n$ whenever $V>0$, contradicting optimality. This proves (i).

It remains to determine the optimized value and its expansion. At the balanced weighting, every edge polynomial $R_j$ equals $A^5B^4/n^8$, so we have
\[
\sum_{j=0}^{q-1}R_jD_j(\bz)=\frac{A^5B^4}{n^8}\Lambda_P(\bz).
\]
By homogeneity, optimizing over the weights of the colors of $P$ at fixed total $T$ gives
\[
\max_{\substack{z_c\ge0\ (c\in C(P))\\\sum_{c\in C(P)}z_c=T}}\Lambda_P(\bz)=T^3\lambda(P).
\]
Substituting the balanced auxiliary weights, define
\[
\Psi_n(x)\coloneqq\max\Biggl\{2TA\left(\frac{B}{n}+c\right)\left(B-\frac{B}{n}+d\right)\prod_{\ell=1}^8t_\ell+2TB\left(\frac{A}{n}+c\right)\left(A-\frac{A}{n}+d\right)\prod_{\ell=1}^8t_\ell+\frac{A^5B^4T^3}{n^8}x\Biggr\},
\]
where the maximum is over $(A,B,c,d,T,t_1,\ldots,t_8)\in\mathbb{S}_{13}$. This is a compact domain. Every maximizing weighting of $\mathcal H_{n,P}$ is constant within each auxiliary class, and every point in this domain can be completed by an optimal weighting of $P$ scaled to total weight $T$. Hence $\lambda(\mathcal H_{n,P})=\Psi_n(\lambda(P))$.

On every compact interval of values of $x$, the objective is uniformly continuous jointly in $x$ and the feasible weights. Hence $\Psi_n$ is continuous. For $x>0$, the test weighting gives $\Psi_n(x)\ge12^{-12}(4+xn^{-8})>4\cdot12^{-12}$. At any maximizer the coefficient of $x$ must be positive, because the base part is at most $4\cdot12^{-12}$. For $0<x<y$, testing a maximizing weighting for $x$ at the parameter $y$ gives $\Psi_n(y)\ge\Psi_n(x)+\frac{A^5B^4T^3}{n^8}(y-x)>\Psi_n(x)$, proving strict monotonicity.

Finally, \eqref{eq:cyclic-variance} and~\eqref{eq:compiler-euler}, together with $\Delta=O(n^{-8})$ and $\eta=O(n^{-8})$, give $A,B=1/12+O(n^{-4})$ and $T=1/12+O(n^{-8})$. Hence $A^5B^4T^3=12^{-12}+O(n^{-4})$. Since the base part is at most $4\cdot12^{-12}$ and $\lambda(P)\le1$, we obtain $\lambda(\mathcal H_{n,P})\le12^{-12}(4+\lambda(P)n^{-8})+O(n^{-12})$. The test weighting in~\eqref{eq:compiler-test} gives the reverse bound $\lambda(\mathcal H_{n,P})\ge12^{-12}(4+\lambda(P)n^{-8})$, proving both the expansion and the nonnegativity of its error in (iii).
\end{proof}

\section{Intervals of uniform densities}\label{sec:interval}
We first prove \cref{thm:seed} by varying $P$ and then $n$ in the preceding construction, and lifting the resulting interval to higher uniformities. We then combine this interval with finite-palette values supplied by ordinary densities to prove \cref{thm:main}.

\begin{proof}[Proof of \cref{thm:seed}]
We first prove the case $r=12$. Fix a sufficiently large $n\in18\N$ with $n\ge36(q+1)$. For $x\in[a,b]$, take the projective sequence of palettes equipped with homomorphisms to $P_*$ from \cref{prop:input}. The palettes $\mathcal H_{n,P_i}$ form a projective sequence because the construction preserves homomorphisms commuting with the maps to $P_*$. \Cref{prop:exact-reader} and continuity give $\lim_{i\to\infty}\lambda(\mathcal H_{n,P_i})=\Psi_n(x)$. \Cref{thm:projective} implies $[L_n,U_n]\subseteq\Pu{12}$, where $L_n\coloneqq\Psi_n(a)$ and $U_n\coloneqq\Psi_n(b)$. Each point in this interval has the stated projective realization. Passing the uniform estimate in \cref{prop:exact-reader}(iii) along each such sequence yields
\[
L_n=12^{-12}(4+an^{-8})+O(n^{-12}),\quad\text{and}\quad U_n=12^{-12}(4+bn^{-8})+O(n^{-12}).
\]
Both endpoints exceed $4\cdot12^{-12}$ and tend to $4\cdot12^{-12}$. Moreover, we have
\[
U_{n+18}-L_n=12^{-12}(b-a)n^{-8}+O(n^{-9})>0
\]
for all sufficiently large $n\in18\N$, and the same estimate holds for $U_n-L_{n+18}$. Thus the consecutive intervals $[L_n,U_n]$ and $[L_{n+18},U_{n+18}]$ overlap for all sufficiently large admissible $n$. The union of $[L_n,U_n]$ over all sufficiently large $n\in18\N$ contains $(4\cdot12^{-12},(4+\eps)12^{-12}]$ for some $\eps>0$.

The palette consisting of the single tuple $(c,c,c_3,\ldots,c_{12})$, where $c,c_3,\ldots,c_{12}$ are pairwise distinct, has Lagrangian $4\cdot12^{-12}$. Its constant sequence realizes the left endpoint.

For every $r\ge12$, the lifting operation in \cref{lem:orientation-lift} gives $r^r\lambda(\operatorname{lift}_{r}(P))=12^{12}\lambda(P)$ and preserves homomorphisms. Applying it to each palette in the realizing sequences gives projective realizations of every point in $[4r^{-r},(4+\eps)r^{-r}]$. By \cref{thm:projective}, these points belong to $\Pu r$. The same $\eps$ works for every $r\ge12$.
\end{proof}

To extend the 12-uniform interval, we need finite-palette values that are dense in higher uniformities. Recall the ordinary Lagrangian $\lambda(H)$ from \cref{sec:palettes}. By the ordinary-density theorem of Li, Liu, and Mubayi~\cite[Theorem 1.1]{ordinary}, we may fix an integer $s_0\ge5$ such that
\begin{equation}\label{eq:ordinary-input}
\Pord s=\{0\}\cup\left[\frac{s!}{s^s},1\right]\qquad\text{for every }s\ge s_0.
\end{equation}
We use this choice of $s_0$ throughout the remaining arguments. The following lemma uses this identity to show that finite-palette Lagrangians are dense in $[s^{-s},1]$ for every $s\ge s_0$.

\begin{lemma}\label{lem:dense-finite-values}
For every $s\ge s_0$, the set $\Du_s\cap[s^{-s},1]$ is dense in $[s^{-s},1]$, and $\{s^{-s},1\}\subseteq\Du_s$.
\end{lemma}
\begin{proof}
For $1\le m\le s!$, \cref{lem:orientation-lift} gives $\lambda(P_m(H))=(m/s!)\lambda(H)$ for each finite simple $s$-graph $H$.
By~\eqref{eq:ordinary-input} and the density of normalized Lagrangians of finite $s$-graphs, we have $[ms^{-s},m/s!]\subseteq\overline{\Du_s}$. Since $s!/s^s<1/2\le m/(m+1)$, we have $m/s!\ge(m+1)s^{-s}$, so consecutive intervals overlap. The union of these intervals over $1\le m\le s!$ is $[s^{-s},1]$. A single rainbow tuple gives $s^{-s}$, and a one-color palette containing its sole tuple gives 1.
\end{proof}

The following products combine the intervals already constructed with these finite-palette values. Let $P$ be a $k$-palette and $Q$ an $s$-palette on disjoint color sets. Define two $(k+s)$-palettes as follows.
\begin{enumerate}
\item The product $P\star Q$ has color set $C(P)\sqcup C(Q)$ and consists of the concatenations of admissible tuples of $P$ and $Q$.
\item The product $P\diamond Q$ has color set $C(P)\times C(Q)$. A tuple $((c_i,d_i))_{i=1}^{k+s}$ is admissible if $(c_1,\ldots,c_k)$ is admissible in $P$ and $(d_{k+1},\ldots,d_{k+s})$ is admissible in $Q$; the unused components are unrestricted.
\end{enumerate}

\begin{lemma}\label{lem:finite-products}
The two products satisfy $\lambda(P\star Q)=\frac{k^ks^s}{(k+s)^{k+s}}\lambda(P)\lambda(Q)$ and $\lambda(P\diamond Q)=\lambda(P)\lambda(Q)$.
If $P'\to P$ and $Q'\to Q$, with $P'$ a $k$-palette and $Q'$ an $s$-palette on disjoint color sets, then $P'\star Q'\to P\star Q$ and $P'\diamond Q'\to P\diamond Q$.
\end{lemma}
\begin{proof}
For the first formula, assign total weight $u$ to $C(P)$ and $1-u$ to $C(Q)$. For fixed $u$, the maximum over weightings with these two total weights is $u^k(1-u)^s\lambda(P)\lambda(Q)$. Maximizing over $u$ gives the optimal split $u=k/(k+s)$.

For the second formula, let $(z_{c,d})$ be a probability weighting and let $x_c\coloneqq\sum_{d\in C(Q)}z_{c,d}$ and $y_d\coloneqq\sum_{c\in C(P)}z_{c,d}$ be its marginals. The two conditions use disjoint sets of tuple positions, so $\Lambda_{P\diamond Q}(\bz)=\Lambda_P(\bx)\Lambda_Q(\by)$. Conversely, any optimal pair of marginals can be simultaneously attained by $z_{c,d}=x_cy_d$.

Given homomorphisms $\theta\colon P'\to P$ and $\phi\colon Q'\to Q$, apply $\theta$ to the colors of $P'$ and $\phi$ to the colors of $Q'$ for the $\star$ product. For the $\diamond$ product, use $(c,d)\mapsto(\theta(c),\phi(d))$. These maps take admissible tuples to admissible tuples in the respective target products.
\end{proof}

We also use the elementary covering identity
\begin{equation}\label{eq:interval-cover}
\bigcup_{d\in D}[Ad,Bd]=[A\ell,Bu]
\end{equation}
whenever $0<A<B$, $D$ is dense in $[\ell,u]\subset(0,\infty)$, and $D$ contains both endpoints. For $t\in(A\ell,Bu)$, the permissible interval $[t/B,t/A]\cap[\ell,u]$ has positive length unless $d=\ell$ or $d=u$ already gives $t\in[Ad,Bd]$; choose $d\in D$ there. Thus a dense set of scaling factors covers every point of the stated interval.

\begin{lemma}\label{lem:terminal-27}
For every $r\ge s_0+3$, we have $[27r^{-r},1]\subseteq\Pu r$.
\end{lemma}
\begin{proof}
Put $s\coloneqq r-3\ge s_0$, and take the projectively realized interval $[\rho,1]$ from \cref{prop:input}. For each $y\in\Du_s\cap[s^{-s},1]$, apply both products in \cref{lem:finite-products} to its finite realizing palette and to each sequence realizing $x\in[\rho,1]$. \Cref{thm:projective} gives the intervals $[\rho y,y]$ and $\frac{27s^sy}{r^r}[\rho,1]$ in $\Pu r$. \Cref{lem:dense-finite-values} and~\eqref{eq:interval-cover} give the intervals $[\rho s^{-s},1]$ and $[27\rho r^{-r},27s^sr^{-r}]$ in $\Pu r$.
These intervals overlap. Indeed, we have
\[
\frac{27s^s}{r^r}=\frac{27}{(s+3)^3}\left(1+\frac{3}{s}\right)^{-s}
>\frac1{(s+3)^3}>s^{-s}.
\]
Here $(1+3/s)^s<e^3<27$, and $s\ge s_0\ge5$ gives the second inequality. Their union contains $[27\rho r^{-r},1]$, and hence the claimed interval.
\end{proof}

\begin{proof}[Proof of \cref{thm:main}]
Let $r\ge s_0+12$ and put $s\coloneqq r-12$. By \cref{lem:dense-finite-values}, the relative Lagrangians $d\coloneqq s^s\lambda(Q)\in[1,s^s]$ of finite $s$-palettes are dense in this interval and include its endpoints. Apply the $\star$ product to the projective sequences constructed in the proof of \cref{thm:seed} for $r=12$ and a fixed finite palette of relative Lagrangian $d$. The identity $r^r\lambda(P\star Q)=(12^{12}\lambda(P))(s^s\lambda(Q))$ shows that the limits of the relative Lagrangians fill $[4d,(4+\eps)d]$. The product preserves the projective maps, so \cref{thm:projective} realizes the corresponding uniform densities. Taking the union over these values of $d$ and applying~\eqref{eq:interval-cover} proves $[4r^{-r},(4+\eps)s^sr^{-r}]\subseteq\Pu r$.
Since $s\ge5$, the right endpoint exceeds $27r^{-r}$. Also $r\ge s_0+3$, so \cref{lem:terminal-27} extends this interval to $[4r^{-r},1]$. \Cref{thm:attainment}(ii) describes the whole set below $4r^{-r}$. Together with \cref{lem:finite-values}, these statements prove \cref{thm:main}.
\end{proof}

\begin{proof}[Proof of \cref{thm:nonclosed}(i) and (ii)]
For (i), \cref{lem:dense-finite-values} and the inclusions $\Du_r\subseteq\Pu r\subseteq\{0\}\cup[r^{-r},1]$ from \cref{lem:finite-values} imply the closure identity for every $r\ge s_0$. \Cref{thm:attainment}(ii) shows that $\Pu r\cap[0,4r^{-r})$ is countable and consists of algebraic numbers. For $r\ge s_0$, the set $\Du_r\cap[r^{-r},4r^{-r})$ is dense in $[r^{-r},4r^{-r})$. For (ii), countability excludes every nondegenerate interval below $4r^{-r}$ for all $r\ge3$. For $1\le C<4$, any interval $[Cr^{-r},1]$ contains such a nondegenerate subinterval, so it cannot be a subset of $\Pu r$.
\end{proof}

Together with \cref{thm:projective}, the proof of \cref{thm:main} shows that, for every $r\ge s_0+12$, $\Pu r$ is exactly the set of limits of the Lagrangians along projective sequences of finite $r$-palettes.

\section{Concluding remarks}\label{sec:concluding}
GPT produced proofs of the following two results for the intermediate uniform density levels. For $1\le j\le r-3$, define $\pi_{j,\mathrm{unif}}(\Fu)$ by replacing $r-2$ with $j$ in the density definitions in \cref{sec:intro}, and write $\Pi_{j,\mathrm{unif},\infty}^{(r)}\coloneqq\{\pi_{j,\mathrm{unif}}(\Fu)\colon\Fu\text{ is a possibly infinite family of }r\text{-graphs}\}$.

A linear ordering of $V(F)$ is \emph{$q$-position-consistent} if every $q$-set contained in two edges occupies the same set of positions when each edge is listed in increasing order.

\begin{theorem}[{\cite[Theorem 1.2]{hierarchy}}]\label{thm:intermediate-zero}
For every $r\ge4$, every $1\le j\le r-3$, and every $r$-graph $F$, we have $\pi_{j,\mathrm{unif}}(F)=0$ if and only if $F$ has a $(j+1)$-position-consistent ordering.
\end{theorem}

\begin{theorem}[{\cite[Theorem 1.5]{hierarchy}}]\label{thm:intermediate-spectrum}
There exists $r_0\in\N$ such that, for all $r\ge r_0$ and $1\le j\le r-3$, we have $\Pi_{j,\mathrm{unif},\infty}^{(r)}=\{0\}\cup[N^{-N},1]$, where $N\coloneqq\binom r{j+1}$.
\end{theorem}

\section*{Declaration on the use of AI}
The authors used generative AI tools to assist in discussing proof strategies, checking proofs, and improving exposition.

\bibliographystyle{abbrv}
\AddToHookNext{cmd/thebibliography/after}{\setlength{\parsep}{1pt}\setlength{\parskip}{1pt}}
\bibliography{uniform_turan_structure}

\newpage 

\appendix
\crefalias{section}{appendix}
\section{Realization by finite forbidden families}\label{app:finite-family-realization}
We extend the finite-family realization theorem of King, Piga, Sales, and Sch\"ulke~\cite[Theorem 1.1]{kpss} from $3$-palettes to $r$-palettes for every $r\ge3$. This supplies the finite-family assertion in \cref{thm:attainment}(ii). Following their rigidity and counting argument~\cite[Sections 5--6]{kpss}, we combine testability of hereditary properties of directed hypergraphs with the palette characterization and separation in \cref{thm:palette-tools}. The proof also gives stability and equality statements for palettes. We include the details needed to handle tuples with repeated colors.

\begin{theorem}\label{thm:app-fr-realization}
For every $r\ge3$ and every finite $r$-palette $P_0$, there is a finite family $\Gu$ of $r$-graphs with $\pi_{r-2}(\Gu)=\lambda(P_0)$. Consequently, $\Du_r\subseteq\Pufin r$.
\end{theorem}

We prove that finitely many local homomorphism constraints bound the Lagrangian by that of a fixed target palette. Fix $r\ge3$. For a finite $r$-palette $Q$, write $d(Q)\coloneqq |Q|/|C(Q)|^r$ for its density under the uniform weighting. A \emph{complete blowup} $S$ of $P$ on a nonempty finite set $V$ is specified by a map $\phi\colon V\to C(P)$, with admissible tuples
\[
S\coloneqq\{(v_1,\ldots,v_r)\in V^r\colon (\phi(v_1),\ldots,\phi(v_r))\in P\}.
\]
The map $\phi$ need not be surjective. If $C(P)=[m]$ and $V_i\coloneqq \phi^{-1}(i)$, then $d(S)=\Lambda_P((|V_i|/|V|)_{i\in[m]})\le\lambda(P)$. All tuple counts in this appendix are ordered and include tuples with repeated colors. Each tuple is counted once.

Call $P$ \emph{Lagrangian-minimal} if $\lambda(P[J])<\lambda(P)$ for every nonempty proper subset $J\subsetneq C(P)$. Every finite palette has a Lagrangian-minimal induced subpalette with the same Lagrangian.

For a finite $r$-palette $P$ and a positive integer $k$, let $\mathcal C_k(P)$ be the class of finite $r$-palettes $Q$ satisfying $Q[J]\to P$ or $Q[J]\to\rev(P)$ for every nonempty $J\subseteq C(Q)$ with $|J|\le k$. The choice between $P$ and $\rev(P)$ may depend on $J$. Moreover, this class is closed under restriction, deletion of tuples, and taking complete blowups. For the last assertion, project at most $k$ colors of a blowup to their image in $Q$, and compose with a homomorphism for that image.

\begin{proposition}\label{prop:app-fr-stability}
Let $P$ be a Lagrangian-minimal $r$-palette with $0<\lambda(P)<1$. There are positive integers $k,N$ and a constant $\delta>0$, depending only on $P$, such that every $Q\in\mathcal C_k(P)$ on $n\ge N$ colors with $d(Q)\ge\lambda(P)-\delta$ admits a complete blowup $S$ of $P$ or $\rev(P)$ on $C(Q)$ satisfying
\[
d(Q)\le d(S)\le\lambda(P),
\quad\text{and}\quad
|Q\triangle S|\le\frac53\bigl(d(S)-d(Q)\bigr)n^r\le\frac53\bigl(\lambda(P)-d(Q)\bigr)n^r.
\]
Every $Q\in\mathcal C_k(P)$, regardless of its number of colors or the value of $d(Q)$, satisfies $\lambda(Q)\le\lambda(P)$.
\end{proposition}
\begin{proof}
Label the colors of $P$ by $[m]$. We first establish rigidity for a fixed finite subpalette of a complete $P$-blowup. Hereditary testability then gives an approximate partition into $m$ parts. Two counting claims compare the tuples violating this partition with the admissible tuples missing from $Q$.

\begin{claim}\label{cl:app-fr-minimal}
The following statements hold.
\begin{enumerate}
\item Every endomorphism of $P$ is an automorphism, and $P\to\rev(P)$ implies $P\cong\rev(P)$.
\item For each $i\ne j$, there is a tuple in $P$ and a position occupied by $i$ such that replacing that occurrence of $i$ by $j$ gives a tuple outside $P$.
\item There are $\beta,\eps_0>0$ such that $y_i\ge\beta$ for every $i\in[m]$ whenever $\by\in\mathbb{S}_m$ and $\Lambda_P(\by)\ge\lambda(P)-\eps_0$.
\end{enumerate}
\end{claim}
\begin{proof}[Proof of the claim]
For (i), a nonsurjective endomorphism has image $J\subsetneq[m]$. Monotonicity of the Lagrangian under homomorphisms would give $\lambda(P)\le\lambda(P[J])<\lambda(P)$, a contradiction. A bijective endomorphism permutes the finite admissible set and is therefore an automorphism. Since $\Lambda_{\rev(P)[J]}=\Lambda_{P[J]}$, the same argument makes any homomorphism $P\to\rev(P)$ surjective. The two palettes have the same number of admissible tuples, so it is an isomorphism.

For (ii), suppose that replacing any one occurrence of $i$ by $j$ always preserves admissibility. Replacing all occurrences in succession then shows that the map sending $i$ to $j$ and fixing all other colors is an endomorphism. This contradicts (i).

For (iii), minimality implies that every optimal weighting has positive coordinates. By compactness of the set of optimal weightings, choose $\beta>0$ so that every such coordinate exceeds $2\beta$. The maximum of $\Lambda_P$ on the compact set $\{\by\in\mathbb{S}_m\colon\min_{i\in[m]}y_i\le\beta\}$ is then strictly less than $\lambda(P)$. Choosing $\eps_0>0$ smaller than this gap ensures that $\Lambda_P(\by)\ge\lambda(P)-\eps_0$ implies $\min_{i\in[m]}y_i>\beta$, as required.
\end{proof}

Take pairwise disjoint sets $U_1,\ldots,U_m$, each of size $rm$, and let $R_*$ be the rainbow part of the complete $P$-blowup on these parts. We use $R_*$ as the root of the extensions below and write $h\coloneqq |C(R_*)|=rm^2$. Thus a tuple of distinct colors belongs to $R_*$ exactly when the tuple of its part indices belongs to $P$. For $u\in C(R_*)$, a \emph{one-color extension at $u$} retains $R_*$, adds one new color $z$, and adds all tuples obtained by replacing $u$ by $z$ in an admissible tuple of $R_*$. Every added tuple contains $z$ exactly once.

\begin{claim}\label{cl:app-fr-rigidity}
Every homomorphism $\theta\colon R_*\to P$ is constant on each $U_i$, and these constant values define an automorphism of $P$. If $u\in U_i$, then in every homomorphism from the one-color extension at $u$ to $P$, the new color has the same image as $U_i$. If $P\not\cong\rev(P)$, then $R_*\nrightarrow\rev(P)$.
\end{claim}
\begin{proof}[Proof of the claim]
For each $i$, at least $r$ members of $U_i$ have the same image $a_i$. Every admissible tuple of $P$, including one with repeated part indices, can be lifted to distinct representatives among these members. Consequently, $i\mapsto a_i$ is an endomorphism of $P$, and hence an automorphism by \cref{cl:app-fr-minimal}(i). Postcomposing $\theta$ with the inverse of this automorphism, assume that $a_i=i$ for all $i$.

Suppose that $u\in U_i$ has image $j\ne i$. Take an arbitrary admissible tuple containing $i$, and lift one occurrence of $i$ to $u$. Lift the other positions to distinct representatives from the subsets of $U_\ell$ whose colors map to $\ell$. Their images show that replacing this occurrence of $i$ by $j$ preserves admissibility. Since the tuple and occurrence were arbitrary, this contradicts \cref{cl:app-fr-minimal}(ii). Thus $\theta$ is constant on each $U_i$.

For a one-color extension at $u\in U_i$, postcompose the homomorphism with an automorphism of $P$ so that every color in $U_\ell$ maps to $\ell$ for each $\ell\in[m]$. If the new color has image $j\ne i$, lift any admissible tuple containing $i$ by using $u$ at one chosen occurrence and distinct original colors at the other positions, avoiding $u$. Replace that occurrence of $u$ by the new color. The resulting tuple belongs to the extension and again contradicts \cref{cl:app-fr-minimal}(ii). There are enough original colors even when the target tuple repeats $i$, since $|U_i|\ge r$.

Finally, the same choice of representatives applied to a homomorphism $R_*\to\rev(P)$ gives $P\to\rev(P)$. \Cref{cl:app-fr-minimal}(i) then gives the last assertion.
\end{proof}

\begin{claim}\label{cl:app-fr-approximation}
For every $\eta>0$, there are positive integers $k_0,N_0$ such that every $Q\in\mathcal C_{k_0}(P)$ on $n\ge N_0$ colors admits a complete blowup $S$ of $P$ or $\rev(P)$ on $C(Q)$ with $|Q\setminus S|\le\eta n^r$.
\end{claim}
\begin{proof}[Proof of the claim]
Delete all tuples of $Q$ with a repeated color, at a cost of at most $\binom r2n^{r-1}$ tuples. Regard the remaining palette as a directed $r$-graph, with its relation defined on ordered tuples of distinct vertices. The property of admitting a homomorphism to $P$ or to $\rev(P)$ is invariant under relabeling and preserved under taking induced subgraphs. Notice that the source tuple has distinct colors, but its image in $P$ may repeat colors.

The testability theorem of Austin and Tao~\cite[Theorem 1.5 in the arXiv version]{austin-tao} supplies $k_0,N_0$ such that, when every restriction to at most $k_0$ vertices has this property, the directed $r$-graph differs in at most $(\eta/2)n^r$ ordered tuples from one that has the property globally. Their edit distance counts unordered $r$-sets on which the induced structure changes; each such set supports at most $r!$ ordered tuples, so tolerance $\eta/2$ gives at most $(\eta/2)r!\binom nr\le(\eta/2)n^r$ ordered-tuple changes.

The defining condition for $\mathcal C_{k_0}(P)$ supplies the required local property. Let $R$ be the modified directed $r$-graph, and let $S$ be the complete blowup defined by a homomorphism from $R$ to $P$ or $\rev(P)$. Since $R\subseteq S$, the rainbow tuples of $Q$ lying outside $S$ are among the edits. Increasing $N_0$ so that $\binom r2n^{r-1}\le(\eta/2)n^r$ proves the claim.
\end{proof}

We now choose the constants for the two counting claims. The root has at most $E\coloneqq h^r$ tuples, so its union with at most $r$ one-color extensions has at most $(r+1)E$ tuples. The constants $c$ and $\gamma$ will be the thresholds in the missing-tuple and degree estimates. Take $\beta,\eps_0$ from \cref{cl:app-fr-minimal}(iii), and put $c\coloneqq (\beta/2)^{r-1}/(2(r+1)E)$ and $\gamma\coloneqq c/(4r)$. Put $D_0\coloneqq rm^{r-1}$ and $p_0\coloneqq (\gamma/(4D_0))^m$. Choose $\alpha>0$ so that $\alpha<\eps_0$ and $E\alpha(2/\beta)^r<\min\{p_0/2,1/2\}$. Set $\delta\coloneqq \alpha/4$. Apply \cref{cl:app-fr-approximation} with $\eta=\delta$, and take $k\ge\max\{k_0,h+r\}$. Finally, choose an integer $N\ge\max\{N_0,k\}$ such that
\begin{equation}\label{eq:app-fr-size}
 \beta N\ge2(h+r),\qquad
 r\binom r2\le\frac{\gamma N}{2},\quad\text{and}\quad
 (r-1)h\le\frac{\gamma N}{4D_0}.
\end{equation}

Let $Q\in\mathcal C_k(P)$ have $n\ge N$ colors and satisfy $d(Q)\ge\lambda(P)-\delta$. \Cref{cl:app-fr-approximation} supplies a complete blowup of $P$ or $\rev(P)$. Replacing $P$ by $\rev(P)$ if necessary, assume that the target of this blowup is $P$. We have $\Lambda_{\rev(P)}=\Lambda_P$, and the preceding constants are unchanged. Among all complete $P$-blowups on $C(Q)$, choose $S$ minimizing $|Q\setminus S|$. Put $B\coloneqq Q\setminus S$ and $A\coloneqq S\setminus Q$, and write the parts of $S$ as $V_1,\ldots,V_m$. Then $|B|\le\delta n^r$ and $|A|=|B|+|S|-|Q|\le2\delta n^r\le\alpha n^r$.
Moreover, we have $d(S)\ge d(Q)-|B|/n^r\ge\lambda(P)-2\delta$. Hence \cref{cl:app-fr-minimal}(iii) gives $|V_i|\ge\beta n$ for every $i\in[m]$.

We will repeatedly choose a uniformly random injective map $\psi\colon C(R_*)\to C(Q)$ with $\psi(U_i)\subseteq V_i$, avoiding a prescribed set of at most $r$ colors. For any fixed ordered tuple of $t\le r$ distinct root colors, the probability of a specified image is at most $(\beta n/2)^{-t}$. Indeed, the denominator is a product of falling factorials, and each factor is at least $\beta n-h-r\ge\beta n/2$. Since $|R_*|\le E$, the probability that some root tuple maps to $A$ is at most
\begin{equation}\label{eq:app-fr-root-failure}
 E|A|\left(\frac{\beta n}{2}\right)^{-r}\le E\alpha\left(\frac{2}{\beta}\right)^r.
\end{equation}

For $x\in C(Q)$, define $d_B(x)\coloneqq|\{(v_1,\ldots,v_r)\in B\colon x\in\{v_1,\ldots,v_r\}\}|$. Thus each tuple is counted once even if $x$ occurs in several positions.

\begin{claim}\label{cl:app-fr-bad-degree}
For every $x\in C(Q)$, we have $d_B(x)\le\gamma n^{r-1}$.
\end{claim}
\begin{proof}[Proof of the claim]
Suppose that $d_B(x)>\gamma n^{r-1}$. For each $i\in[m]$, change only the part assigned to $x$, placing it in $V_i$. Let $B_i(x)$ be the tuples of $Q$ containing $x$ that lie outside the resulting blowup. The choice of $S$ gives $|B_i(x)|\ge d_B(x)>\gamma n^{r-1}$. There are at most $r\binom r2n^{r-2}$ tuples containing $x$ and a repeated color. This bound follows by specifying a position occupied by $x$ and a pair of positions with equal entries. By~\eqref{eq:app-fr-size}, at least $(\gamma/2)n^{r-1}$ members of each $B_i(x)$ are rainbow.

For each $i\in[m]$, classify the rainbow tuples in $B_i(x)$ by the position of $x$ and the part indices of the other $r-1$ entries. There are at most $D_0$ types. Hence, for every $i$, choose a set $\mathcal B_i\subseteq B_i(x)$ of one type with $|\mathcal B_i|\ge(\gamma/(2D_0))n^{r-1}$. The corresponding tuple of part indices lies outside $P$ when $x$ is assigned to part $i$. Reserve $r-1$ root colors in the corresponding parts $U_j$ for this type, using disjoint reserved colors for different $i$. There are $(r-1)m$ reserved colors in total, while each $U_j$ has $rm$ colors.

Choose $\psi$ as above, avoiding $x$, and expose these reserved groups in order. Suppose that the previously exposed images form a set $W$. Then $|W|\le h$, and at most $(r-1)|W|n^{r-2}$ members of $\mathcal B_i$ meet $W$. By~\eqref{eq:app-fr-size}, at least $(\gamma/(4D_0))n^{r-1}$ candidates remain. Conditional on the previously exposed images, the current group is uniformly distributed over its permissible ordered images, of which there are at most $n^{r-1}$. Its conditional probability of giving a tuple in $\mathcal B_i$, together with $x$, is therefore at least $\gamma/(4D_0)$. This bound holds for every previous exposure, so all $m$ groups succeed with probability at least $p_0$.

By~\eqref{eq:app-fr-root-failure} and the choice of $\alpha$, the probability that some admissible tuple of $R_*$ is mapped outside $Q$ is less than $p_0/2$. Thus some such injection maps every admissible tuple of $R_*$ into $Q$, and the images of the $m$ reserved groups give tuples in $\mathcal B_1,\ldots,\mathcal B_m$, respectively, when combined with $x$.

Consider the subpalette of $Q$ on $\psi(C(R_*))\cup\{x\}$ consisting of the image of $R_*$ and the $m$ selected tuples from the sets $\mathcal B_i$. It has $h+1\le k$ colors but admits no homomorphism to $P$. Indeed, \cref{cl:app-fr-rigidity} fixes the root labels up to an automorphism. After postcomposing a homomorphism with the inverse of the induced automorphism of $P$, each possible image $i$ of $x$ is excluded by the selected tuple from $\mathcal B_i$. There is no homomorphism to $\rev(P)$ either. If the targets are not isomorphic, the root already excludes such a homomorphism; otherwise, compose it with an isomorphism to $P$. Any homomorphism from the induced palette on these colors would also be a homomorphism from this subpalette. This contradicts $Q\in\mathcal C_k(P)$.
\end{proof}

We next give a lower bound on the number of missing tuples associated with each $b\in B$. For $b\in B$, let $X$ be the set of colors occurring in $b$, and let $A_1(b)$ consist of the rainbow tuples in $A$ containing exactly 1 color from $X$.

\begin{claim}\label{cl:app-fr-missing}
For every $b\in B$, we have $|A_1(b)|\ge cn^{r-1}$.
\end{claim}
\begin{proof}[Proof of the claim]
Write $X=\{x_1,\ldots,x_\ell\}$, where $\ell\le r$. For each $j$, choose an original root color in $U_i$ where $x_j\in V_i$, and form its one-color extension with new color $z_j$. Take the union of these extensions with their common root $R_*$. Do not add other tuples containing 2 distinct new colors. Finally, add the tuple obtained from $b$ by replacing every occurrence of $x_j$ by $z_j$.

Call the resulting palette $T_b$. It has at most $h+r\le k$ colors. By \cref{cl:app-fr-rigidity}, any homomorphism $T_b\to P$ can be postcomposed with an automorphism of $P$ so that each $U_i$ maps to $i$; it must then send $z_j$ to the part label of $x_j$. The image of the tuple corresponding to $b$ therefore lies outside $P$, since $b\notin S$. As in the preceding claim, there is no homomorphism to $\rev(P)$. Consequently, there is no injective homomorphism $T_b\to Q$. Indeed, its image $J$ would have at most $k$ colors, so $Q\in\mathcal C_k(P)$ would give $Q[J]\to P$ or $Q[J]\to\rev(P)$. Composing with $T_b\to Q[J]$ would contradict the preceding exclusions.

Fix $z_j\mapsto x_j$ and choose a uniformly random injection of the root mapping each $U_i$ into $V_i$, avoiding $X$. The tuple corresponding to $b$ maps to $b\in Q$. Every other tuple of $T_b$ maps into $S$ and is rainbow; its image meets $X$ in 0 or 1 color. Since no such injection is a homomorphism, the image of at least one of these tuples lies outside $Q$. By \eqref{eq:app-fr-root-failure}, with probability greater than $1/2$ every admissible tuple of $R_*$ is mapped into $Q$. Thus with probability at least $1/2$ some tuple of $T_b$ maps to $A_1(b)$.

Besides the tuple corresponding to $b$, $T_b$ has at most $(r+1)E$ tuples. Fix $a\in A_1(b)$ and a tuple of $T_b$ containing one new color. Its new color has a prescribed image, and its other $r-1$ colors are distinct root colors. The probability that its image is $a$ is at most $(\beta n/2)^{-(r-1)}$. A union bound gives $\frac12\le |A_1(b)|(r+1)E(\beta n/2)^{-(r-1)}$, which proves the claim. The construction also applies when $b$ repeats colors, since it uses one new color for each member of $X$, with repeated occurrences using that same new color. All the other added tuples still contain exactly 1 new color.
\end{proof}

By \cref{cl:app-fr-bad-degree,cl:app-fr-missing}, double counting the pairs $(a,b)$ with $b\in B$ and $a\in A_1(b)$ gives
\[
 cn^{r-1}|B|
 \le\sum_{b\in B}|A_1(b)|
 \le r\gamma n^{r-1}|A|.
\]
For the upper bound, each tuple $a\in A$ has at most $r$ colors, and at most $\gamma n^{r-1}$ tuples of $B$ contain any one of them. Thus $|A|\ge4|B|$, so $|S|-|Q|=|A|-|B|\ge0$ and we have
\[
 |Q\triangle S|=|A|+|B|
 \le\frac53(|A|-|B|)
 =\frac53\bigl(d(S)-d(Q)\bigr)n^r.
\]
Since $d(S)\le\lambda(P)$, this proves the stability assertion.

To prove $\lambda(Q)\le\lambda(P)$ for every $Q\in\mathcal C_k(P)$, suppose instead that some such $Q$ has $\lambda(Q)>\lambda(P)$. Choose a rational probability weighting $\bx\in\mathbb{S}_{C(Q)}$ with $\Lambda_Q(\bx)>\lambda(P)$. Complete blowups of $Q$ in these proportions, on arbitrarily many colors, belong to $\mathcal C_k(P)$ and have density $\Lambda_Q(\bx)$ under the uniform weighting. These blowups satisfy the hypotheses of the density assertion just proved, which bounds their densities by $\lambda(P)$, a contradiction.
\end{proof}

\begin{proof}[Proof of \cref{thm:app-fr-realization}]
If $\lambda(P_0)=0$, forbid a single-edge $r$-graph. If $\lambda(P_0)=1$, use the empty forbidden family. Otherwise choose a Lagrangian-minimal induced subpalette $P\subseteq P_0$ with $\lambda(P)=\lambda(P_0)$, and take $k$ from \cref{prop:app-fr-stability}.

Up to relabeling, there are only finitely many palettes $R$ on at most $k$ colors with $R\nrightarrow P$ and $R\nrightarrow\rev(P)$. For each one, \cref{thm:palette-tools}(ii) gives an $r$-graph $F_R\in\Col(R)\setminus\Col(P)$. Let $\Gu$ consist of these witnesses. It is finite, and $P$ is $\Gu$-deficient. Each $F_R$ has an edge, since every edgeless $r$-graph is $P$-colorable.

If $Q$ is a finite $\Gu$-deficient palette, every induced subpalette on at most $k$ colors maps to $P$ or to $\rev(P)$. Otherwise that subpalette colors its corresponding $F_R$, and hence so does $Q$. Thus $Q\in\mathcal C_k(P)$, and \cref{prop:app-fr-stability} gives $\lambda(Q)\le\lambda(P)$. The characterization in \cref{thm:palette-tools}(i), together with the deficiency of $P$, now gives $\pi_{r-2}(\Gu)=\lambda(P)=\lambda(P_0)$.
\end{proof}

Note that the palette $P_0$ need not be $\Gu$-deficient. Passing to $P$ preserves the Lagrangian, but may change the class of colorable $r$-graphs.

In the case $0<\lambda(P_0)<1$, the realizing family can be chosen as a finite truncation of the canonical forbidden family of the minimal target $P$. Choose an integer $M$ at least the number of vertices of every member of $\Gu$, and let $\Fu_{\le M}(P)$ consist of one representative of each isomorphism class of $r$-graphs $F$ with $|V(F)|\le M$ and $F\notin\Col(P)$. This finite family contains $\Gu$, while $P$ and its blowups remain deficient. Hence $\pi_{r-2}(\Fu_{\le M}(P))=\lambda(P)$. This truncation concerns the canonical family associated with $P$; it need not be a subfamily of a previously given forbidden family.

We next extend the stability estimate to arbitrary weightings. Its equality cases determine the optimal supports, and the unweighted estimate determines the palettes with the maximum number of admissible tuples. Fix a Lagrangian-minimal palette $P$ on $[m]$ with $0<\lambda(P)<1$. Let $k,N,\delta$ be as in \cref{prop:app-fr-stability}, and let $\Gu$ be the finite family constructed from $P$ in the preceding proof. When a blowup of $Q$ is approximated by a blowup of $P$ or $\rev(P)$, the clones of one color of $Q$ may lie in several parts of that blowup. The numbers $z_{v,i}$ below record these portions of its weight.

\begin{samepage}
\begin{proposition}\label{prop:app-fr-consequences}
The following statements hold.
\begin{enumerate}
\item Let $Q\in\mathcal C_k(P)$ and let $\bx\in\mathbb{S}_{C(Q)}$ satisfy $\Lambda_Q(\bx)\ge\lambda(P)-\delta/2$. There are $P'\in\{P,\rev(P)\}$ and nonnegative numbers $z_{v,i}$, with $\sum_{i\in[m]}z_{v,i}=x_v$ for each $v\in C(Q)$, such that
\begin{equation}\label{eq:app-fr-split}
 \sum_{\substack{\boldsymbol v\in C(Q)^r\\\boldsymbol i\in[m]^r}}
 \left|\ind{\boldsymbol v\in Q}-\ind{\boldsymbol i\in P'}\right|
 \prod_{a\in[r]}z_{v_a,i_a}
 \le\frac53\bigl(\lambda(P)-\Lambda_Q(\bx)\bigr).
\end{equation}
\item Let $Q\in\mathcal C_k(P)$ satisfy $\lambda(Q)=\lambda(P)$, and let $\bx$ be an optimal weighting with support $J$. Then $Q[J]$ is a complete blowup of $P$ or $\rev(P)$, with all $m$ parts nonempty. If $\bx$ has minimum support among all optimal weightings of $Q$, then $Q[J]$ is isomorphic to $P$ or $\rev(P)$, and $|J|=m$.
\item For every sufficiently large $n$, every palette maximizing the number of admissible tuples among all $n$-color $\Gu$-deficient palettes is a complete blowup of $P$ or $\rev(P)$.
\end{enumerate}
\end{proposition}
\end{samepage}
\begin{proof}
For (i), approximate $\bx$ by rational weightings $\bx^{(j)}$, and take complete $Q$-blowups $Q_j$ in these proportions, with $n_j\coloneqq |C(Q_j)|$ chosen so that $\lim_{j\to\infty}n_j=\infty$. They belong to $\mathcal C_k(P)$ and satisfy $d(Q_j)=\Lambda_Q(\bx^{(j)})$, so $\lim_{j\to\infty}d(Q_j)=\Lambda_Q(\bx)$. For all sufficiently large $j$, we may therefore apply \cref{prop:app-fr-stability} and let $S_j$ be the resulting complete blowup of $P$ or $\rev(P)$. Pass to a subsequence with a fixed target $P'$. If the parts of $S_j$ are $V_1^{(j)},\ldots,V_m^{(j)}$, put $z_{v,i}^{(j)}\coloneqq|\{u\in V_i^{(j)}\colon u\text{ is a clone of }v\}|/n_j$ for $v\in C(Q)$ and $i\in[m]$. Pass to a convergent subsequence of this finite-dimensional array. Its limit satisfies $\sum_{i=1}^m z_{v,i}=\lim_{j\to\infty}\sum_{i=1}^m z_{v,i}^{(j)}=\lim_{j\to\infty}x_v^{(j)}=x_v$ for every $v\in C(Q)$.

At $\bz^{(j)}$, the left side of~\eqref{eq:app-fr-split} is exactly $|Q_j\triangle S_j|/n_j^r$. Indeed, for fixed $\boldsymbol v,\boldsymbol i$, the number of ordered tuples of colors in $Q_j$ whose $a$th entry is a clone of $v_a$ in $V_{i_a}^{(j)}$ for every $a\in[r]$ is $\prod_{a\in[r]}(n_jz_{v_a,i_a}^{(j)})$. This count also holds when entries repeat, since the same color of $Q_j$ may occupy more than one position. Passing to the limit in the estimate from~\cref{prop:app-fr-stability} proves the result.

For (ii), apply (i). The right side of \eqref{eq:app-fr-split} is zero. For every $v\in J$, choose a label $\phi(v)$ with $z_{v,\phi(v)}>0$. For each $(v_1,\ldots,v_r)\in J^r$, including tuples with repeated colors, we have $\prod_{a\in[r]}z_{v_a,\phi(v_a)}>0$. All summands in~\eqref{eq:app-fr-split} are nonnegative, so $(v_1,\ldots,v_r)\in Q\Longleftrightarrow(\phi(v_1),\ldots,\phi(v_r))\in P'$. Thus $Q[J]$ is a complete blowup of $P'$. Since $J$ supports an optimal weighting of $Q$, we have $\lambda(Q[J])=\lambda(Q)=\lambda(P)$. The map $\phi$ is therefore surjective. Otherwise, monotonicity of the Lagrangian under homomorphisms and minimality would give $\lambda(Q[J])\le\lambda(P'[\phi(J)])<\lambda(P)$.

If, in addition, $\bx$ has minimum support, no part can contain two positive-weight colors. Indeed, the polynomial of a complete blowup depends only on the total weight in each part, so concentrating that weight on one color preserves the polynomial value and reduces the size of the support. Thus each part contains exactly 1 positive-weight color, proving the minimum-support assertion in (ii).

For (iii), choose an optimal weighting $\by\in\mathbb{S}_m$ of $P$. For each $n$, choose nonnegative integers $n_1,\ldots,n_m$ summing to $n$ with $|n_i/n-y_i|\le1/n$ for every $i\in[m]$. The vector $(n_i/n)_{i\in[m]}$ differs from $\by$ by at most $m/n$ in the $\ell^1$ norm. The difference-of-products bound in the proof of \cref{thm:compactness} therefore gives $\Lambda_P((n_i/n)_{i\in[m]})\ge\lambda(P)-rm/n$. This is the density of the complete $P$-blowup with part sizes $n_1,\ldots,n_m$. All complete blowups of $P$ or $\rev(P)$ are $\Gu$-deficient. A $\Gu$-deficient palette $Q$ with the maximum number of tuples therefore belongs to $\mathcal C_k(P)$ and satisfies the density hypothesis of \cref{prop:app-fr-stability} for sufficiently large $n$. Let $S$ be the complete blowup supplied by that proposition. Since $S$ is also $\Gu$-deficient, the choice of $Q$ gives $d(Q)\ge d(S)$. The two inequalities in \cref{prop:app-fr-stability} imply $d(Q)=d(S)$ and $Q=S$.
\end{proof}

In particular, for all sufficiently large $n$, the maximum number of admissible tuples in an $n$-color $\Gu$-deficient palette is
\[
 \max_{\substack{n_1,\ldots,n_m\in\mathbb Z_{\ge0}\\n_1+\cdots+n_m=n}}
 \sum_{(i_1,\ldots,i_r)\in P}\prod_{a\in[r]}n_{i_a}.
\]
Reversal preserves the polynomial at every vector of part sizes, so $P$ and $\rev(P)$ give the same maximum.

The conclusions of \cref{prop:app-fr-stability,prop:app-fr-consequences} also hold for palettes deficient for the finite truncation $\Fu_{\le M}(P)$ defined after the proof of \cref{thm:app-fr-realization}, with $M\ge\max\{|V(F)|\colon F\in\Gu\}$. Indeed, this truncation contains $\Gu$, and all complete blowups of $P$ or $\rev(P)$ remain deficient. All stability and equality statements in this appendix concern palettes, rather than their host $r$-graphs.

\section{Countable palettes and uniform density}\label{app:countable-colorability}
We give a countable palette with Lagrangian $2/9$ that colors every finite $3$-graph. Its admissible tuples are closed under all coordinate permutations.

Let $P$ have color set $C(P)\coloneqq\binom{\N}{2}$, viewed as the edge set of the countable complete graph. A triple of colors is admissible if its entries are the 3 edges of a triangle, in any order. For each $m\ge3$, let $P_m\coloneqq P[\binom{[m]}2]$. Thus $P_3\subseteq P_4\subseteq\cdots$ and $P=\bigcup_{m\ge3}P_m$.

\begin{proposition}\label{prop:countable-colorability}
The following statements hold.
\begin{enumerate}
\item For every $m\ge3$, we have $\lambda(P_m)=\lambda(P)=2/9$.
\item Every finite $3$-graph is $P$-colorable, but for each fixed $m\ge3$, all sufficiently large complete $3$-graphs fail to be $P_m$-colorable.
\end{enumerate}
\end{proposition}
\begin{proof}
For (i), write $x_{ij}$ for the weight of the color $\{i,j\}$. Since each triangle gives 6 admissible tuples, the polynomial of $P_m$ is
\[
\Lambda_{P_m}(\bx)=6\sum_{1\le i<j<k\le m}x_{ij}x_{ik}x_{jk}.
\]
Assigning weight $1/3$ to each edge of one triangle gives $\lambda(P_m)\ge2/9$. For the reverse bound, choose an optimal probability weighting with minimum support. If 2 disjoint edges have positive weight, no triangle contains both, so the polynomial is affine when weight is transferred between them. Concentrating their total weight on one of the 2 edges cannot decrease the value and reduces the support, a contradiction. Thus the edges of positive weight are pairwise intersecting. Since the polynomial value is positive, they contain a triangle. Any other edge is disjoint from at least one edge of this triangle, so the support consists of exactly its 3 edges. AM--GM therefore gives $\lambda(P_m)\le6(1/3)^3=2/9$. Every finite subset of $C(P)$ is contained in $\binom{[m]}2$ for some $m$, so the definition of the countable-palette Lagrangian gives $\lambda(P)=\sup_{m\ge3}\lambda(P_m)=2/9$.

For (ii), label the vertices of a finite $3$-graph by $[n]$ and color each pair by itself, that is, $\chi(\{i,j\})\coloneqq\{i,j\}$. For every edge $\{i<j<k\}$, the face colors in the prescribed order are $(\{j,k\},\{i,k\},\{i,j\})$, which form an admissible tuple of $P$. Hence every finite $3$-graph is $P$-colorable.

In contrast, for a fixed $m\ge3$, Ramsey's theorem~\cite[Theorem B]{ramsey} implies that every coloring of the pairs of a sufficiently large vertex set with the $\binom m2$ colors of $P_m$ contains a monochromatic triangle. Its 3 face colors form a tuple $(e,e,e)$, which is not admissible in $P_m$. Thus no vertex order and pair coloring can make the complete $3$-graph on this vertex set $P_m$-colorable.
\end{proof}

In particular, $P$ colors complete $3$-graphs of arbitrarily large order, so the supremum of uniform densities attainable by $P$-colorable $3$-graphs is 1, although $\lambda(P)=2/9$. By contrast, the values at the uniform weightings on the finite truncations satisfy $|P_m|/|C(P_m)|^3=6\binom m3/\binom m2^3=8(m-2)/(m^2(m-1)^2)$, and hence $\lim_{m\to\infty}|P_m|/|C(P_m)|^3=0$.

\section{Products and projective sequences}\label{sec:products}
We extend the palette products from \cref{sec:interval} to arbitrary uniform densities and then consider projective sequences below $4r^{-r}$.

\begin{theorem}\label{thm:general-products}
Let $k,s\ge3$, $a\in\Pu k$, and $b\in\Pu s$. Then $ab\in\Pu{k+s}$ and $\frac{k^ks^sab}{(k+s)^{k+s}}\in\Pu{k+s}$.
\end{theorem}

To prove \cref{thm:general-products}, we first remove the vertex order from two ordered witnesses. We use the argument of Lin, Sun, Wang, and Zhou~\cite[Lemma 4.5]{lswz}, with the same pair of witnesses excluding every palette in a possibly infinite collection. For completeness, we give the proof of this form below. For an ordered hypergraph $H$, write $H^{\rm rev}$ for the same hypergraph with its vertex order reversed.

\begin{lemma}\label{lem:order-amplification}
Let $r\ge3$, let $R$ be a finite $r$-palette, and let $\mathscr T$ be any collection of $r$-palettes. Suppose there are finite ordered $r$-graphs $H^+,H^-$, each containing an edge, that are $R$-colorable, such that $H^+$ is not $T$-colorable and $H^-$ is not $\rev(T)$-colorable for any $T\in\mathscr T$. Then there is a finite unordered $r$-graph $G\in\Col(R)$ that does not belong to $\Col(T)$ for any $T\in\mathscr T$.
\end{lemma}
\begin{proof}
Add isolated vertices so that both ordered $r$-graphs have $v$ vertices, and put $m\coloneqq(v-1)^2+1$. For sufficiently large $N$, a greedy packing gives at least $\eta_mN^2$ subsets of $[N]$, each of size $m$, any two meeting in at most 1 vertex, where $\eta_m>0$ depends only on $m$. Indeed, each selected $m$-set rules out at most $\binom m2\binom{N-2}{m-2}$ candidates, out of $\binom Nm$ in total.

Independently color every $(r-1)$-set of $[N]$ uniformly from $C(R)$, and retain every $r$-set whose tuple in the natural order belongs to $R$. The resulting unordered $r$-graph is $R$-colorable.

Fix an order $\prec$ of $[N]$. In each packed $m$-set, the monotone-subsequence theorem~\cite[p.~467]{erdosszekeres} gives $v$ vertices whose natural order agrees with $\prec$ or reverses it. Choose this subset deterministically from the two orders. Fix an $R$-coloring of each witness, extending it arbitrarily to faces outside its edges. On the chosen vertices in their natural order, require the fixed coloring of $H^+$ in the agreeing case and that of $H^-$ in the reversing case. Each prescribed coloring occurs with probability $p\coloneqq|C(R)|^{-\binom v{r-1}}>0$. On the corresponding event, the selected vertices contain $H^+$ or $H^-$, respectively, in the natural order. In order $\prec$, the latter is $(H^-)^{\rm rev}$, which cannot be $T$-colored because $H^-$ cannot be $\rev(T)$-colored. Either event therefore excludes every $T\in\mathscr T$ at once.

The events for distinct packed sets use disjoint $(r-1)$-faces because $r-1\ge2$, so they are independent. The failure probability for this order is at most $(1-p)^{\eta_mN^2}$. Taking a union bound over all $N!$ orders gives total failure probability at most $N!(1-p)^{\eta_mN^2}<1$ for sufficiently large $N$. Some resulting $r$-graph therefore excludes every target palette in every vertex order.
\end{proof}

\begin{proof}[Proof of \cref{thm:general-products}]
If $a=0$ or $b=0$, both conclusions give zero. Assume that $a,b>0$. For a $(k+s)$-palette $R$ and $\varnothing\ne I=\{i_1<\cdots<i_t\}\subseteq[k+s]$, its \emph{coordinate projection} is the $t$-palette $R_I\coloneqq\{(c_{i_1},\ldots,c_{i_t})\colon(c_1,\ldots,c_{k+s})\in R\}$ on $C(R)$. Each projected tuple is included only once. We first check how a coloring of a coordinate projection extends after common vertices are added to every edge.

\begin{claim}\label{clm:cone-projection}
Let $R$ be a finite $(k+s)$-palette, and let $R_I$ be its projection onto the first $k$ coordinates. If an ordered $k$-graph $F$ is $R_I$-colorable, then the ordered $(k+s)$-graph obtained by adding the same $s$ new vertices to every edge, placing all new vertices last, is $R$-colorable. The analogous statement holds for the last $s$ coordinates by placing $k$ new vertices first.
\end{claim}
\begin{proof}[Proof of the claim]
For each edge $e$ of $F$, choose an admissible $R$-tuple completing its prescribed projected tuple. If $Z$ is the set of new vertices, the faces $(e\setminus\{v\})\cup Z$ use the original face colors. They agree whenever shared by different edges. The other faces are $e\cup(Z\setminus\{z\})$. Such a face determines both $e$ and $z$, since the old and new vertex sets are disjoint, so its completion color can be chosen independently for each edge. This gives a consistent coloring. Similarly, for projection onto the last $s$ coordinates, place the new vertices first.
\end{proof}

Choose forbidden families $\Fu,\Gu$ with densities $a,b$. By~\cref{thm:palette-tools}(i), the finite $\Fu$-deficient palettes have Lagrangians with supremum $a$, and the finite $\Gu$-deficient palettes have Lagrangians with supremum $b$. Both deficiency conditions are preserved when palettes are reversed.

We first prove that $ab\in\Pu{k+s}$. Let $\mathscr T_\diamond$ consist of all products $P\diamond Q$ with $P$ a finite $\Fu$-deficient palette and $Q$ a finite $\Gu$-deficient palette. Its Lagrangians have supremum $ab$. For every finite $(k+s)$-palette $R$ with $\lambda(R)>ab$, we will construct a finite $(k+s)$-graph that is $R$-colorable but is not colorable by any member of $\mathscr T_\diamond$.

Put $I\coloneqq[k]$ and $J\coloneqq[k+s]\setminus I$, and let $R_I,R_J$ be the projections of $R$ onto these position sets. For every probability weighting $\bx$ on $C(R)$, we have $\Lambda_R(\bx)\le\Lambda_{R_I}(\bx)\Lambda_{R_J}(\bx)$.
The pointwise inequality follows by expanding the product, since every tuple in $R$ contributes one of its terms. Taking maxima gives $\lambda(R)\le\lambda(R_I)\lambda(R_J)$. Hence $\lambda(R_I)>a$ or $\lambda(R_J)>b$. In the first case some $F\in\Fu$ is $R_I$-colorable. Choose a witnessing order and apply \cref{clm:cone-projection}. The resulting ordered $(k+s)$-graph is $R$-colorable but cannot be colored by any $P\diamond Q\in\mathscr T_\diamond$. Indeed, such a coloring would induce a $P$-coloring of $F$ by taking the first component of each color on the faces in the first coordinate block. The second case uses a member of $\Gu$ and places the new vertices first.

We also need a witness against all reversed targets. Interchanging the components of each color gives $\rev(P\diamond Q)\cong\rev(Q)\diamond\rev(P)$. Apply the same argument with the blocks of sizes $s,k$ and the deficiency classes interchanged. The threshold remains $ab$, so it supplies another ordered $R$-colorable witness excluding every reversed target. \Cref{lem:order-amplification} gives a finite unordered $(k+s)$-graph $G_R$ that is $R$-colorable and is not colorable by any member of $\mathscr T_\diamond$.

Let $\Hu_\diamond$ consist of one such $G_R$ for every finite palette $R$ with $\lambda(R)>ab$, taking representatives on finite labeled color sets. Every target product is deficient for this family, so~\cref{thm:palette-tools}(i) gives density at least $ab$. Every palette with larger Lagrangian colors its designated $G_R$, giving the reverse bound. Hence $\pi_{k+s-2}(\Hu_\diamond)=ab$.

For the product on disjoint color sets, consider the products $P\star Q$ of a finite $\Fu$-deficient palette $P$ and a finite $\Gu$-deficient palette $Q$. Their Lagrangians have supremum $\frac{k^ks^sab}{(k+s)^{k+s}}$. Fix a finite $(k+s)$-palette $R$ with larger Lagrangian, and put $C_I\coloneqq\{c_i\colon(c_1,\ldots,c_{k+s})\in R,\ i\in I\}$ and $C_J\coloneqq\{c_j\colon(c_1,\ldots,c_{k+s})\in R,\ j\in J\}$, the sets of colors occurring in the two coordinate blocks. We construct the ordered witness in two cases.
\begin{enumerate}
\item Suppose $C_I\cap C_J\ne\varnothing$. Choose tuples in which a common color occurs at positions $p\in I$ and $q\in J$. The ordered edges $[k+s+1]\setminus\{p\}$ and $[k+s+1]\setminus\{q+1\}$ share a face at positions $q$ and $p$, respectively. Assign to the first edge the tuple with the common color at position $q$, and to the second the tuple with that color at position $p$. Every other face belongs to just one edge, so the assignment gives an $R$-coloring. No $P\star Q$ colors this ordered $(k+s)$-graph, because its two coordinate blocks use disjoint color sets.
\item Suppose $C_I\cap C_J=\varnothing$. After deleting unused colors, $R_I$ and $R_J$ have disjoint color sets $C_I$ and $C_J$, and $R$ is a subpalette of $R_I\star R_J$. These deletions preserve all three Lagrangians, so \cref{lem:finite-products} gives $\lambda(R)\le\frac{k^ks^s}{(k+s)^{k+s}}\lambda(R_I)\lambda(R_J)$. Thus $\lambda(R_I)>a$ or $\lambda(R_J)>b$. In the first case, choose an $R_I$-colorable $F\in\Fu$ and apply \cref{clm:cone-projection} to a witnessing order. The resulting ordered $(k+s)$-graph is $R$-colorable but cannot be colored by any of the products $P\star Q$, since projection onto the first $k$ positions would give a $P$-coloring of $F$. In the second case, use an $R_J$-colorable member of $\Gu$ and place the new vertices first; projection onto the last $s$ positions gives the analogous contradiction.
\end{enumerate}
Finally, we have
$\rev(P\star Q)=\rev(Q)\star\rev(P)$ after relabeling the disjoint color sets, so the same construction works for the reversed targets. Apply \cref{lem:order-amplification} to obtain an unordered $R$-colorable witness excluding every product $P\star Q$ under consideration. Form a forbidden family from one such witness for each finite $(k+s)$-palette $R$ with $\lambda(R)>\frac{k^ks^sab}{(k+s)^{k+s}}$. By~\cref{thm:palette-tools}(i), this family has density $\frac{k^ks^sab}{(k+s)^{k+s}}$.
\end{proof}

\begin{proposition}\label{prop:projective-stabilizes}
If finite $r$-palettes satisfy $P_{n+1}\to P_n$ and $\lim_{n\to\infty}\lambda(P_n)=t<4r^{-r}$, then there is $N\in\N$ such that $\lambda(P_n)=t$ for every $n\ge N$.
\end{proposition}
\begin{proof}
Choose $N$ with $\lambda(P_N)<4r^{-r}$, so $P_N$ is rainbow. For $n\ge N$, composition gives a homomorphism $\theta_n\colon P_n\to P_N$. For each $c\in C(P_N)$, every admissible tuple of $P_n$ has at most 1 entry in $\theta_n^{-1}(c)$, so $\Lambda_{P_n}$ is affine in the weights within this set when all other weights are fixed. Starting from an optimal weighting, successively concentrate the total weight of each nonempty set $\theta_n^{-1}(c)$ on 1 color without decreasing the polynomial. Thus $\lambda(P_n)$ is attained on at most $|C(P_N)|$ colors. Only finitely many palette Lagrangians occur on that many colors. Since $(\lambda(P_n))_{n\ge N}$ is nonincreasing with limit $t$, increasing $N$ if necessary gives $\lambda(P_n)=t$ for every $n\ge N$.
\end{proof}

\end{document}